\documentclass{tac}

\usepackage{amsmath}
\usepackage{amsfonts}
\usepackage{mathrsfs}
\usepackage{amssymb}
\usepackage{tikz-cd}
\usepackage{hyperref}
\usepackage{bbm}
\usepackage{color}
\usepackage{tensor}
\usepackage{tipa}
\usepackage{bussproofs}
\usepackage{ stmaryrd }
\usepackage{appendix}
\usepackage{ tipa }
\usepackage{quiver}
\usepackage{adjustbox}

\author{William Troiani}

\address{School of Mathematics and Statistics, University of Melbourne\\
	Parkville 3000. Melbourne, Vic, Australia\\[5pt]
}

\title{The internal logic and finite colimits}

\copyrightyear{2022}

\keywords{Topos, finite colimits, internal language, type theory, logic.}

\amsclass{18A15, 18N45}

\eaddress{william.a.troiani@gmail.com}

\newcommand{\bb}[1]{\mathbb{#1}}

\newcommand{\call}[1]{\mathcal{#1}}
\newcommand{\und}[1]{\underline{\hspace{#1 cm}}}
\newcommand{\adj}[1]{\text{\textopencorner}{#1}\text{\textcorner}}

\newcommand{\lto}{\longrightarrow}

\newenvironment{scprooftree}[1]%
{\gdef\scalefactor{#1}\begin{center}\proofSkipAmount \leavevmode}%
	{\scalebox{\scalefactor}{\DisplayProof}\proofSkipAmount \end{center} }

\begin{document}
	
	\maketitle
	
	\begin{abstract}
		We describe how finite colimits can be described using the internal lanuage, also known as the Mitchell-Benabou language, of a topos, provided the topos admits countably infinite colimits. This description is based on the set theoretic definitions of colimits and coequalisers, however the translation is not direct due to the differences between set theory and the internal language, these differences are described as \emph{internal} versus \emph{external}. Solutions to the hurdles which thus arise are given.
	\end{abstract}
	
	
	\tableofcontents
	
	\section{Introduction}
	Type theories were originally suggested by Russell and Whitehead \cite{russell} as a foundation for mathematics, for a historical view, see \cite[p.124]{lambekscott}. For an introduction to type theories, see \cite[\S 3.3]{su}. The idea of approaching logic from the perspective of category theory is originally due to Lawvere. In the modern form, associated to every elementary topos $\call{E}$ are several type theories. We will consider a fixed such type theory and refer to it as the \emph{internal logic}. This type theory can be used to describe constructions of objects and morphisms in $\call{E}$, and also to prove things about them. For example, if $f,g$ are morphisms in $\call{E}$ both with domain $A$ and both with codomain $B$, then if the following is derivable in the internal logic:
	\begin{equation}
		\vdash \forall a:A,\text{ }fa = ga
	\end{equation}
	then the morphisms $f,g$ are equal. More generally, the internal logic is \emph{sound} in that any formula which is provable inside the internal logic represents some true statement about the topos in question (see \cite[\S D1.3.2, D4.2.3, D4.3.3, D4.4.6]{Johnstone} for a discussion on soundness). In this paper we show how the type theory can be used to describe finite coproducts \ref{sec:coproducts} and coequalisers \ref{sec:coequaliser}, which then leads to a description of finite colimits \ref{sec:colimits}. We remark that one must also describe an initial object, but this is easy and is done in Section \ref{sec:initial_object}.
	
	The internal logic offers many syntactical similarities with ZF set theory, in fact, due to the existence of the internal logic, elementary topoi are often referred to as ``generalised theories of sets" \cite[\S Prologue, page 1]{MM}. Hence, our approach will be to use the standard set-theoretic descriptions of finite coproducts and of coequalisers to form terms in the interal logic whose interpretations (ie, associated objects/morphisms) will be finite coproducts and coequalisers respectively \emph{of the original topos}. Here, we must observe a subtlety: recall the standard construction of the coequaliser $\operatorname{Coeq}(f,g)$ of two functions of sets $f,g:A \lto B$: first we define for each $n > 0$ a relation $R_n$ on $B$ as follows, a pair $(b,b')$ of elements in $B$ is in $R_n$ if and only if there exists a sequence $(a_0,...,a_n)$ of elements in $A$ such that for each $i = 0,...,n-1$ we either have $f(a_i) = g(a_{i+1})$ or $f(a_{i+1}) = g(a_i)$. Then, the Coequaliser $\operatorname{Coeq}(f,g)$ is given by the quotient $B/\cup_{n = 1}^\infty R_n$. Our method for mimicking this description using the internal logic will require for each object $C$ in our topos $\call{E}$ that the category $\operatorname{Sub}(C)$ of subobjects of $C$ admits countably infinite coproducts. A circumstance where this holds is when $\call{E}$ admits countably infinite coproducts, hence we assume this is true of the topoi that we work with.
	
	Despite what has been said, to describe finite colimits using the internal logic of a topos, it is not as simple as directly translating the standard constructions in ZF, as there is a subtle but significant difference between the interal logic and ZF set theory. In ZF set theory, the set constructor union is \emph{extensional} and in the internal logic it is \emph{intensional}. We now explain what we mean by this point along with our solution.
	
	Let $A,B$ be sets. Recall that the disjoint union $A \sqcup B$ can be defined by first constructing sets $A^\ast,B^\ast$ which are defined by the following rules.
	\begin{equation}
		(a,1) \in A^\ast \text{ iff }a \in A,\qquad (b,2) \in B^\ast \text{ iff } b \in B
	\end{equation}
	Then, the disjoint union is given by taking the union $A^\ast \cup B^\ast$ of these sets. Translating this into the internal logic, we run into trouble because we can only take the union of two subobjects \emph{of the same object}. That the union of \emph{any} pair of sets can be constructed in ZF is what we mean by the union being \emph{extensional}. That the union of two subobjects via the internal logic can only be taken once they have each been embedded into a common object is what we mean by the union being \emph{intensional}.
	
	Our solution is given precisely by Corollary \ref{cor:embedded_subobject}, but we give the general idea here. Continuing with the same sets $A,B$, we consider the sets, each subsets of the set $\call{P}A \times \call{P}B$.
	\begin{align}
		\lbrace A \rbrace &:= \big\lbrace (\lbrace a \rbrace,\varnothing) \in \call{P}A \times \call{P}B \mid a \in A \rbrace\\
		 \lbrace B \big\rbrace &:= \big\lbrace (\varnothing, \lbrace b \rbrace) \in \call{P}A \times \call{P}B \mid b \in B \big\rbrace
	\end{align}
	We notice that $A \cong \lbrace A \rbrace$ and $B \cong \lbrace B \rbrace$. A coproduct of $A,B$ can then be taken to be the set
	\begin{equation}
		A \coprod B := \lbrace A \rbrace \cup \lbrace B \rbrace
	\end{equation}
	along with the maps $A \lto A \coprod B,\text{ }a \longmapsto (\lbrace a \rbrace, \varnothing)$ and $B \lto A \coprod B,\text{ }b \longmapsto (\varnothing, \lbrace b \rbrace)$.
	
	We have already mentioned the standard construction of the coequaliser of two functions, we note here that this construction does \emph{not} make use of union at any point, hence the translation of this construction into the internal logic is direct, and will not use Corollary \ref{cor:embedded_subobject}. Indeed, the main difficulty in describing finite coproduct using the interal language lies in describing finite colimits, whereas the description of coequalisers is straight forward.
	
	\section{Finite colimits in the topos of sets}\label{sec:in_sets}
	As mentioned in the Introduction, the union operator in ZF is \emph{extensional} whereas the union operator in the internal logic is \emph{intensional}. This means that the standard construction of the disjoint union of two sets cannot be directly translated into the internal logic. We show in this Section our general method for constructing finite colimits but in the special case of the topos of sets. We begin with the construction of binary coproducts.
	
	Let $A,B$ be sets, we consider the following subset of the product of power sets $\call{P}A \times \call{P}B$. We suggestively denote this set $A \coprod B$.
	\begin{align}
		\begin{split}
			\label{eq:bin_disjoint_union}A \coprod B := \big\lbrace (X,Y) \in \call{P}A \times \call{P}B &\mid (\exists a \in A, X = \lbrace a \rbrace \text{ and }Y = \varnothing)\\
		&\text{ or } (\exists b \in B, X = \varnothing \text{ and }Y = \lbrace b \rbrace )\big\rbrace
		\end{split}
	\end{align}
	We notice that this subset is in bijection with the standard construction of the coproduct of $A$ and $B$, ie, the disjoint union.
	\begin{align*}
		A \sqcup B &\lto A \coprod B\\
		a &\longmapsto (\lbrace a \rbrace, \varnothing)\\
		b &\longmapsto (\varnothing, \lbrace b \rbrace)
	\end{align*}
	Consider now an arbitrary finite set $J = \lbrace X_1,...,X_n\rbrace$ of sets. In a standard way we can construct the disjoint union of the $n$ elements of $J$ from the binary disjoint union just considered. To ease notation we let $\prod_{i = 1}^n\call{P}X_i, \coprod_{i = 1}^n X_i$ denote respectively 
	\begin{equation*}
		\big(\hdots (\call{P}X_1 \times \call{P}X_2)\times \hdots \times\call{P}X_n\big)
		\end{equation*}
	and
	\begin{equation*}
		\big(\hdots (X_1 \coprod X_2) \coprod \hdots \coprod X_n\big)
		\end{equation*}
	Similarly for future objects. We have the following.
	\begin{equation}
		\begin{tikzcd}
			\prod_{i=1}^n \call{P}X_1\\
			\coprod_{i = 1}^n X_i\arrow[u,rightarrowtail]
		\end{tikzcd}
	\end{equation}
	Moving towards a finite colimit, we now consider a finite set $\lbrace f_1,...,f_m\rbrace$ of functions where each $f_i$ has domain $\operatorname{dom}f_i$ and codomain $\operatorname{cod}f_i$ with $\operatorname{dom}f_i, \operatorname{cod}f_i \in J$. We define two morphisms, both with domain $\prod_{i = 1}^m \call{P}\operatorname{dom}f_i$ and codomain $\prod_{i = 1}^n \call{P}X_i$. The first of these morphisms is easy to describe, it is the product of the morphisms $\operatorname{id}_{\call{P}\operatorname{dom}f_i}$ which we denote by $f$. The second morphism is more complicated. In general, any function $g: D \lto E$ induces a function $\call{P}f: \call{P}D \lto \call{P}E$ which given a subset $C \subseteq D$ applies the function $f$ to each element $d \in D$ to obtain a subset $f(D) \in \call{P}E$. We consider the product of each of these morphisms $\call{P} f_i$ (for $i = 1,...,m$) and denote this product $h$.  Both $f$ and $h$ descend to morphisms $g_0,g_1: \coprod_{i = 1}^m \operatorname{dom}f_i\lto \coprod_{i = 1}^n X_i$ where $g_0$ is the descendent of $f$ and $g_1$ is the descendent of $h$. We now have two commuting diagrams, the first is the square given by removing the bottom most functions from both the top and bottom rows of Diagram \ref{eq:product_diag_set}, the second is the square given by removing the top most functions.
	\begin{equation}\label{eq:product_diag_set}
		\begin{tikzcd}
			\prod_{i = 1}^m \call{P}\operatorname{dom}f_i\arrow[r,shift left, "{f}"]\arrow[r,shift right, swap, "{h}"] & \prod_{i  =1}^n \call{P} X_i\\
			\coprod_{i = 1}^m \operatorname{dom}f_i\arrow[r,shift left, "{g_0}"]\arrow[r,shift right, swap, "{g_1}"]\arrow[u,rightarrowtail] & \coprod_{i = 1}^n X_i\arrow[u,rightarrowtail]
		\end{tikzcd}
	\end{equation}
	In general, how do we describe the coequaliser of two functions $l_1,l_2: D \lto E$? First we let $\call{R}$ denote the smallest equivalence relation on $E$ so that for all $d \in D$, we have $(l_1(d),l_2(d)) \in \call{R}$. Then, given $e \in E$ we denote by $[e]$ the equivalence class represented by $e$. The standard construction of the coequaliser is then given as follows.
	\begin{equation}
		\operatorname{Coeq}(l_1,l_2) := \lbrace [e] \mid e \in E\rbrace
	\end{equation}
	This is an element of $\call{P}\call{P}E$.
	
	The coequaliser should also admit a morphism $b: \coprod_{i = 1}^n X_i \lto \operatorname{Coeq}(g_0,g_1)$. This will be given be arguing that there exists a morphism $d: \prod_{i = 1}^n \call{P}X_i \lto \call{P}\Big(\prod_{i =1 }^n \call{P}X_i\Big)$ which descend to $b$. Towards this end, we define the following.
	\begin{align*}
		d: \prod_{i = 1}^n\call{P}X_i &\lto \call{P}\Big(\prod_{i = 1}^n \call{P}X_i\Big)\\
		Z &\longmapsto [Z]
	\end{align*}
	Indeed the function $d$ descends to a function $b$ as required. We can now have a commutative Diagram \eqref{eq:com_diag_set} which is completed by the following function.
	\begin{align*}
		c: \operatorname{Coeq}(g_0,g_1) &\lto \call{P}\Big(\prod_{i = 1}^n \call{P}X_i\Big)\\
		[(z_1,\hdots,z_n)] &\longmapsto [(\lbrace z_1 \rbrace, \hdots, \lbrace z_n \rbrace)]
	\end{align*}
	The point is that by beginning with a colimit Diagram (the bottom row of \eqref{eq:com_diag_set}) we can construct the morphisms in the top row of \eqref{eq:com_diag_set} which descend to the original Diagram. It is this top row which will offer a simple translation into the internal logic. The rest of the work will be to prove that the rest of the Diagram behaves as expected.
	\begin{equation}\label{eq:com_diag_set}
		\begin{tikzcd}[column sep = huge]
			\prod_{i = 1}^m \call{P}\operatorname{dom}f_i\arrow[r,shift left, "{h_0}"]\arrow[r,shift right,swap, "{h_1}"] & \prod_{i = 1}^n \call{P}X_i\arrow[r,"{d}"] & \call{P}\Big(\prod_{i = 1}^n \call{P}X_i\Big)\\
			\coprod_{i = 1}^m \operatorname{dom}f_i\arrow[u,rightarrowtail]\arrow[r,shift left, "{g_0}"]\arrow[r,shift right, swap, "{g_1}"] & \coprod_{i = 1}^n X_i\arrow[u,rightarrowtail]\arrow[r] & \operatorname{Coeq}(g_0,g_1)\arrow[u,rightarrowtail, "{c}"]
		\end{tikzcd}
	\end{equation}
	This rather innocuous description of finite colimits is then generalised to arbitrary topoi in Section \ref{sec:arbitrary_topos}. In Section \ref{sec:type_theory} we introduce the type theory and in Section \ref{sec:crucial_lemma} we present some complications and our solutions which arise when translating the description of finite colimits given in this Section to the more general setting of an elementary topos admitting countably infinite colimits.
	
	\section{Type theory}\label{sec:type_theory}
	Type theories were originally suggested as a foundation for mathematics. In this paper, a type theory will be a convenient language of subobjects and morphisms inside a topos. We begin with the definition of a type theory and then specialise to those arising from topoi. In particular, Definition \ref{def:type_theory} will make no mention of any topoi.
	\begin{definition}
		\label{def:type_theory}
		A \textbf{type theory} consists of the following.
		\begin{itemize}
			\item A class of \textbf{types}, including special types $\Omega, \mathbbm{1}$. Also, for each type $\tau$, there is a countably infinite set of \textbf{variables} of type $\tau$.
			\item A class of \textbf{function symbols} $f: \tau \to \sigma$, where $f$ is a formal symbol, $\tau$ and $\sigma$ are types.
			\item A class of \textbf{relation symbols} $R \subseteq \tau$, where $R$ is a formal symbol, and $\tau$ is a type.
			\item A class of \textbf{terms}, where to each term $t$, there is an associated type, $\tau$. The notation $t : \tau$ means ``$t$ is of type $\tau$". Also, there is an associated set $\operatorname{FV}(t)$ of \textbf{free variables}.
			\item A class of \textbf{formulas}, where similarly to terms, to each formula there is an associated set of free variables, however unlike terms, there is no associated type.
			\item For every finite sequence $\Delta = (x_1:\tau_1,...,x_n:\tau_n)$ of variables, a binary relation $\vdash_\Delta$ of \textbf{entailment} between formulas whose free variables appear in $\Delta$. We write $p \vdash_\Delta q$ when the pair $(p,q)$ is in the relation $\vdash_\Delta$, and such an expression $p \vdash_\Delta q$ is called a \textbf{sequent}. The sequence $\Delta$ in a sequent $p \vdash_\Delta q$ is the \textbf{context} and a context is \textbf{valid} if the free variables of $p$ and $q$ are a subset of the underlying set of $\Delta$.
		\end{itemize}
		This data is required to be subject to the following.
		\begin{enumerate}
			\item If $\tau$ and $\sigma$ are types, then so are $\tau \times \sigma$ and $P\tau$. Identification between types is also allowed.\\\\
			The following axiom describes how the class of terms and formulas along with their associated type (for terms) and free variable sets are constructed. First a class of \textbf{preterms} will be defined by induction, then appropriate equivalence classes of preterms will constitute the terms, a similar process will be undertaken for formulas.
			\item The class of \textbf{preterms} satisfy the following.
			\begin{enumerate}
				\item There is the class of \textbf{atomic preterms}, containing all the variables, as well as the special term $\ast : \mathbbm{1}$. There is also assumed to be \textbf{atomic preformulas} which are $\top, \bot$.
				\item The preterms and preformulas are closed under the following \textbf{preterm and preformula formation rules},
				\begin{enumerate}
					\item If $t:\tau$ and $s:\sigma$ are preterms, then $\langle t,s\rangle : \tau \times \sigma$ is a preterm.
					\item If $t: \tau \times \sigma$ is a preterm, then $\operatorname{fst}(t) : \tau$ and $\operatorname{snd}(t) : \sigma$ are preterms.
					\item If $f:\tau \to \sigma$ is a function symbol, and $t:\tau$ is a preterm, then $ft:\sigma$ is a preterm.
					\item If $R \subseteq \tau$ is a relation symbol, and $t:\tau$ is a term, then $R(t)$ is a preformula.
					\item If $t,s:\tau$, then $t = s$ is a preformula.
					\item If $p$ and $q$ are preformulas, then $p \wedge q$, $p \vee q$, and $p \Rightarrow q$, are all preformulas.
					\item If $\lbrace p_i\rbrace_{i = 0}^\infty$ is a countable set of preformulas, then $\bigvee_{i = 0}^\infty p_i$ is a preformula.
					\item If $x : \tau$ is a variable, and $p$ is a preformula, then $\forall x : \tau, p$ and $\exists x : \tau, p$ are both preformulas.
					\item If $p$ is a preformula, and $x : \tau$ is a variable, then $\lbrace x : \tau \mid p\rbrace$ is a preterm of type $P\tau$.
					\item If $t : \tau$ and $T : P\tau$, then $t \in T$ is a preformula.
				\end{enumerate}
				\item The preterms and preformulas have free variable sets satisfying the following.
				\begin{enumerate}
					\item $\operatorname{FV}(\ast) = \operatorname{FV}(\bot) = \operatorname{FV}(\top) = \varnothing$.
					\item FV$(x:\tau) = \lbrace x:\tau \rbrace$, if $x:\tau$ is any variable.
					\item $\operatorname{FV}(\langle t,s\rangle) = \operatorname{FV}(t) \cup \operatorname{FV}(s)$.
					\item $\operatorname{FV}(\operatorname{fst}(t)) = \operatorname{FV}(\operatorname{snd}(t)) = \operatorname{FV}(t)$.
					\item $\operatorname{FV}(ft) = \operatorname{FV}(t)$.
					\item $\operatorname{FV}(R(t)) = \operatorname{FV}(t)$.
					\item $\operatorname{FV}(t = s) = \operatorname{FV}(t) \cup \operatorname{FV}(s)$.
					\item $\operatorname{FV}(t \in T) = \operatorname{FV}(t) \cup \operatorname{FV}(T)$.
					\item $\operatorname{FV}(p \wedge q) = \operatorname{FV}(p \vee q) = \operatorname{FV}(p \Rightarrow q) = \operatorname{FV}(p) \cup \operatorname{FV}(q)$.
					\item $\operatorname{FV}(\bigvee_{i = 0}^\infty p_i) = \bigcup_{i = 0}^\infty\operatorname{FV}(p_i)$
					\item $\operatorname{FV}(\lbrace x : \tau \mid p\rbrace) = \operatorname{FV}(\forall x : \tau,p) = \operatorname{FV}(\exists x : \tau,p) = \operatorname{FV}(p) \setminus \lbrace x \rbrace$.
				\end{enumerate}
				An \textbf{occurence} of a variable $x$ in a term will mean any which is not the $``x : \tau"$ part of any term of the form $\lbrace x : \tau\mid p\rbrace$, $\forall x : \tau, p$, or $\exists x : \tau, p$. Any occurence of a variable which is not part of a term's free variable set is a \textbf{bound variable}.
				\item The terms and formulas are $\alpha$-$\text{equivalence}$ classes of preterms and preformulas respectively, where \textbf{$\alpha$-equivalence} is the smallest relation on the collection of preterms or preformulas respectively subject to the following.
				\begin{enumerate}
					\item If $t =_\alpha s$, then $t$ and $s$ are of the same type,
					\item If $t =_\alpha t'$ and $s =_\alpha s'$, then $\langle t,s\rangle =_\alpha \langle t', s'\rangle$,
					\item If $t =_\alpha s$, then $\operatorname{fst}(t) =_\alpha \operatorname{fst}(t)$, and $\operatorname{snd}(t) =_\alpha \operatorname{snd}(s)$,
					\item If $t =_\alpha t'$, then $ft =_\alpha ft'$,
					\item If $t =_\alpha t'$, then $R(t) =_\alpha R(t')$,
					\item If $t =_\alpha t'$ and $s =_\alpha s'$, then $(t = s) =_\alpha (t' = s')$,
					\item If $p =_\alpha p'$ and $q =_\alpha q'$, then $p \wedge q =_\alpha p' \wedge q'$, $p \vee q =_\alpha p' \vee q'$, and $p \Rightarrow q =_\alpha p' \Rightarrow q'$,
					\item If for all $i$, $p_i =_\alpha p_i'$, then $\bigvee_{i = 0}^\infty p_i =_\alpha \bigvee_{i = 0}^\infty p_i'$,
					\item If $p =_\alpha p'$, then $\lbrace x : \tau\mid p\rbrace =_\alpha \lbrace x : \tau \mid p'\rbrace$, $\forall (x : \tau)p =_\alpha \forall (x : \tau)p'$, and $\exists (x : \tau)p =_\alpha \exists(x : \tau)p'$,
					\item If $t =_\alpha t'$ and $T =_\alpha T'$, then $t \in T =_\alpha t' \in T'$,
					\item $\lbrace x : \tau \mid p\rbrace =_\alpha \lbrace y : \tau \mid p[x:=y]\rbrace$,\\
					$\forall x : \tau,p =_\alpha \forall y : \tau,(p[x:= y])$,\\
					$\exists x : \tau,p =_\alpha \exists y : \tau,(p[x :=y])$, provided that no free occurrence of $x$ in $p$ is such that $y$ in place of $x$ would be bound. In the above, the notation $p[x:=y]$ means the term $p$ but with every occurrence of $x$ replaced by $y$.
				\end{enumerate}
				\item The free variable set of a term $[t]$ is $\text{FV}(t)$, where $t$ is any representative of $[t]$, and similarly for formulas. For convenience, equivalence class brackets will be dropped.
				\item There is also the following shorthand notation,
				\begin{enumerate}
					\item $\neg p$ means $p \Rightarrow \bot$,
					\item $p \Leftrightarrow q$ means $(p \Rightarrow q) \wedge (q \Rightarrow p)$,
					\item $\lbrace x \rbrace$ means $\lbrace x' : \tau \mid x = x'\rbrace$, where $x : \tau$,
					\item $\varnothing_{x:\tau}$ means $\lbrace x:\tau \mid \bot\rbrace$
				\end{enumerate}
			\end{enumerate}
			\item Finally, there are two sets of axioms concerning the entailment relation, these are as follows.
			\begin{enumerate}
				\item The \textbf{structural rules}:
				\begin{enumerate}
					\item\label{rule:ax} For any variable $p$ and valid context $\Delta$.
					\begin{prooftree}
						\AxiomC{$p \vdash_\Delta p$}
					\end{prooftree}
					\item\label{rule:cut} For any variables $p,q,r$ and valid context $\Delta$.
					\begin{prooftree}
						\AxiomC{$p \vdash_\Delta q$}
						\AxiomC{$q \vdash_\Delta r$}
						\BinaryInfC{$p \vdash_\Delta r$}
					\end{prooftree}
					\item\label{rule:sub} For any sequence of variables $(x_1:\tau_1,...,x_n:\tau_n)$, sequence of terms $(t_1:\tau_1,...,t_n:\tau_n)$, and context $\Sigma$ such that each variable which appears in $\Delta$, except for those in $(x_1,...,x_n)$, also appear in $\Sigma$,
					\begin{prooftree}
						\AxiomC{$p \vdash_{\Delta} q$}
						\UnaryInfC{$p[(x_1,...,x_n):= (t_1,...,t_n)] \vdash_{\Sigma} q[(x_1,...,x_n):= (t_1,...,t_n)]$}
					\end{prooftree} where $p[(x_1,...,x_n):= (t_1,...,t_n)]$ means the term given by $p$ after simultaneously substituting each $x_i$ for $t_i$ (similarly for $q$). It is assumed that no free variable in any $t_i$ becomes bound in $p[(x_1,...,x_n):= (t_1,...,t_n)]$ nor $q[(x_1,...,x_n):= (t_1,...,t_n)]$. This can always be achieved by renaming bound variables. \footnote{This final rule is the one to which contexts owe their existence. The essential point is that from $p \vdash_{\Delta,x:X} q$ (where $\Delta,x:\tau$ is the context given by appending $x:\tau$ to the end of $\Delta$), one can infer that $p[x := t] \vdash_\Delta q[x := t]$ only if there exists a term $t:X$ such that $\text{FV}(t) \subseteq \Delta$. Lambek and Scott point out \cite[\S II.1 p.131]{lambekscott} that to deduce $\forall x : X, p\vdash \exists x:X, p$ from $\forall x : X, p\vdash_{x:X} \exists x:X, p$ without there existing a closed term of type $X$ is undesirable from a logical point of view, as although ``for all unicorns $x$, $x$ has a horn", it is not the case that ``there exists a unicorn $x$, such that $x$ has a horn", because (presumably) there does not exist any unicorns at all!}\\\\
					Note: this axiom also allows introducing superfluous variables to the context $\Delta$, and also for rearrangement of elements.
				\end{enumerate}
				\item The \textbf{logical rules}, in the following, the notation $\Delta,x:\tau$ means the context given by $\Delta$ with the variable $x:\tau$ appended to the end,
				\begin{enumerate}
					\item True and False.\label{rule:true_false}
					\begin{center}
						\AxiomC{}
						\RightLabel{$(\top)$}
						\UnaryInfC{$p \vdash_\Delta \top$}
						\DisplayProof
						\qquad
						\AxiomC{}
						\RightLabel{$(\bot)$}
						\UnaryInfC{$\bot \vdash_{\Delta} p$}
						\DisplayProof
					\end{center}
					\item\label{rule:conjunction} Conjunction introduction and elimination.
					\begin{center}
						\AxiomC{$r \vdash_{\Delta} p$}
						\AxiomC{$r \vdash_{\Delta} q$}
						\RightLabel{$(\wedge I)$}
						\BinaryInfC{$r \vdash_{\Delta} p \wedge q$}
						\DisplayProof
						\qquad
						\AxiomC{$r \vdash_{\Delta} p \wedge q$}
						\RightLabel{$(\wedge E)_L$}
						\UnaryInfC{$r \vdash_{\Delta} p$}
						\DisplayProof
						\qquad
						\AxiomC{$r \vdash_{\Delta} p \wedge q$}
						\RightLabel{$(\wedge E)_R$}
						\UnaryInfC{$r \vdash_{\Delta} q$}
						\DisplayProof
					\end{center}
					\item\label{rule:disjunction} Disjunction introduction and elimination.
					\begin{center}
						\AxiomC{$p \vdash_{\Delta} r$}
						\AxiomC{$q \vdash_{\Delta} r$}
						\RightLabel{$(\vee I)$}
						\BinaryInfC{$p \vee q \vdash_{\Delta} r$}
						\DisplayProof
						\qquad
						\AxiomC{$p \vee q \vdash_{\Delta} r$}
						\RightLabel{$(\vee E)_L$}
						\UnaryInfC{$p \vdash_{\Delta} r$}
						\DisplayProof
						\qquad
						\AxiomC{$p \vee q \vdash_{\Delta} r$}
						\RightLabel{$(\vee E)_R$}
						\UnaryInfC{$q \vdash_{\Delta} r$}
						\DisplayProof
					\end{center}
					\item\label{rule:implication} Implication introduction and elimination.
					\begin{center}
						\AxiomC{$p \wedge q \vdash_{\Delta} r$}
						\RightLabel{$(\Rightarrow I)$}
						\UnaryInfC{$p \vdash_{\Delta} q \Rightarrow r$}
						\DisplayProof
						\qquad
						\AxiomC{$p \vdash_{\Delta} q \Rightarrow r$}
						\RightLabel{$(\Rightarrow E)$}
						\UnaryInfC{$p \wedge q \vdash_{\Delta} r$}
						\DisplayProof
					\end{center}
					\item\label{rule:universal} For all introduction and elimination. In what follows, $x:\tau$ is arbitrary.
					\begin{center}
						\AxiomC{$p \vdash_{\Delta,x:\tau} q$}
						\RightLabel{$(\forall I)$}
						\UnaryInfC{$p \vdash \forall x: \tau, q$}
						\DisplayProof
						\qquad
						\AxiomC{$p \vdash_{\Delta} \forall x : \tau, q$}
						\RightLabel{$(\forall E)$}
						\UnaryInfC{$p \vdash_{\Delta, x:\tau} q$}
						\DisplayProof
					\end{center}
					\item\label{rule:existential} Existential introduction and elimination. 
					\begin{center}
						\AxiomC{$p\vdash_{x:\tau, \Delta}q$}
						\RightLabel{$(\exists I)$}
						\UnaryInfC{$\exists x:\tau, p \vdash_{\Delta} q$}
						\DisplayProof
						\qquad
						\AxiomC{$\exists x: \tau, p\vdash_{\Delta} q$}
						\RightLabel{$(\exists E)$}
						\UnaryInfC{$p \vdash_{x:\tau, \Delta} q$}
						\DisplayProof
					\end{center}
					\item\label{rule:equality} Equality introduction. In what follows, $x:\tau$ is arbitrary.
					\begin{center}
						\AxiomC{}
						\RightLabel{$(=I)$}
						\UnaryInfC{$\top \vdash x = x$}
						\DisplayProof
						\end{center}
					As well as equality elimination:
					\begin{center}
						\AxiomC{}
						\RightLabel{$(=E)$}
						\UnaryInfC{$x_1 = y_1 \wedge \hdots \wedge x_n = y_n \wedge p \vdash_{\Delta} p[(x_1,...,x_n) := (y_1,...,y_n)]$}
						\DisplayProof
					\end{center}
				\end{enumerate}
				We write $t \vdash s$ for $t \vdash_\varnothing s$ and $\vdash_\Delta p$ for $\top \vdash_\Delta p$.
			\end{enumerate}
		\end{enumerate}
		The logical rules and structural rules together form the \textbf{deduction rules}.
	\end{definition}
	\begin{remark}
		Notice that Definition \ref{def:type_theory} admits no deduction rules associated to terms of the form $\lbrace x:\tau \mid p\rbrace$ nor to terms of the form $\bigvee_{i = 0}^\infty p_i$. Such rules do exist but are given as \emph{derived rules} later on (Proposition \ref{prop:derived_rules}) once the internal logic of a topos has been defined.
	\end{remark}
	
	\subsection{Some helpful type theory lemmas}\label{sec:helpful_lemmas}
	\label{helpfulLemmas}
	Many reasonable sounding statements concerning the entailment relation do in fact follow from the axioms and deduction rules, this section provides a collection of particularly helpful ones. Many of these will be used in Section \ref{sec:arbitrary_topos}. In what follows, $\Delta,x:\tau$ will always mean the context given by appending $x:\tau$ to the end of the sequence $\Delta$, we give full proofs to illustrate the basic methods involved in working with type theories.
	\begin{lemma}
		\label{weakening}
		Let $p,q,r$ be arbtirary formulas and $x:\tau$ an arbitrary variable.
		\begin{enumerate}
			\item $p \wedge q \vdash_\Delta r$ if and only if $q \wedge p \vdash_\Delta r$
			\item $p \vdash_\Delta \neg(\neg p)$,
			\item if $p \vdash_\Delta q$, then $p \wedge r \vdash_\Delta q$,
			\item $\neg q \vdash_\Delta \neg(q \wedge p)$,
			\item $p \wedge (q \vee r) \vdash_\Delta (p \wedge q) \vee (p \wedge r)$,
			\item $(p \vee q) \wedge \neg q \vdash_\Delta p$, and
			\item $(\exists x:\tau, p) \wedge q \vdash_{\Delta}\exists x:\tau, p \wedge q$ and $\exists x:\tau, p \wedge q \vdash_{\Delta} (\exists x:\tau, p) \wedge q$
		\end{enumerate}
	\end{lemma}
	\begin{proof}
		\begin{enumerate}
			\item There is the following proof tree.
			\begin{prooftree}
				\AxiomC{$q \wedge p \vdash_\Delta q \wedge p$}
				\RightLabel{$\eqref{rule:conjunction}$}
				\UnaryInfC{$q \wedge p \vdash_\Delta p$}
				\AxiomC{$q \wedge p \vdash_\Delta q \wedge p$}
				\RightLabel{$\eqref{rule:conjunction}$}
				\UnaryInfC{$q \wedge p \vdash_\Delta q$}
				\RightLabel{$\eqref{rule:conjunction}$}
				\BinaryInfC{$q \wedge p \vdash_\Delta p \wedge q$}
				\AxiomC{$\vdots$}
				\noLine
				\UnaryInfC{$p \wedge q \vdash_\Delta r$}
				\RightLabel{$\eqref{rule:cut}$}
				\BinaryInfC{$q \wedge p \vdash_\Delta r$}
			\end{prooftree}
			\item Let $\pi$ denote the following proof tree.
			\begin{scprooftree}{0.9}
				\AxiomC{$p \wedge (p \Rightarrow \bot) \vdash_\Delta p \wedge (p \Rightarrow \bot)$}
				\RightLabel{$\eqref{rule:conjunction}$}
				\UnaryInfC{$p \wedge (p \Rightarrow \bot) \vdash_\Delta p \Rightarrow \bot$}
				\AxiomC{$p \wedge (p \Rightarrow \bot) \vdash_\Delta p \wedge (p \Rightarrow \bot)$}
				\RightLabel{$\eqref{rule:conjunction}$}
				\UnaryInfC{$p \wedge (p \Rightarrow \bot) \vdash_\Delta p$}
				\RightLabel{$\eqref{rule:conjunction}$}
				\BinaryInfC{$p \wedge (p \Rightarrow \bot) \vdash_\Delta (p \Rightarrow \bot) \wedge p$}
			\end{scprooftree}
			Then, there is the following proof tree.
			\begin{prooftree}
				\AxiomC{$\pi$}
				\noLine
				\UnaryInfC{$\vdots$}
				\noLine
				\UnaryInfC{$p \wedge (p \Rightarrow \bot) \vdash_\Delta (p \Rightarrow \bot) \wedge p$}
				\AxiomC{$p \Rightarrow \bot \vdash_\Delta p \Rightarrow \bot$}
				\RightLabel{$\eqref{rule:implication}$}
				\UnaryInfC{$(p \Rightarrow \bot) \wedge p \vdash_\Delta \bot$}
				\RightLabel{$\eqref{rule:cut}$}
				\BinaryInfC{$p \wedge (p \Rightarrow \bot) \vdash_\Delta \bot$}
				\RightLabel{$\eqref{rule:implication}$}
				\UnaryInfC{$p \vdash_\Delta(p \Rightarrow \bot) \Rightarrow \bot$}
			\end{prooftree}
			\item Observe the following proof tree.
			\begin{prooftree}
				\AxiomC{$\vdots$}
				\noLine
				\UnaryInfC{$p \vdash_\Delta q$}
				\AxiomC{$q \wedge r \vdash_\Delta q \wedge r$}
				\RightLabel{$\eqref{rule:implication}$}
				\UnaryInfC{$q \vdash_\Delta r \Rightarrow (q \wedge r)$}
				\RightLabel{$\eqref{rule:cut}$}
				\BinaryInfC{$p \vdash_\Delta r \Rightarrow (q \wedge r)$}
				\RightLabel{$\eqref{rule:implication}$}
				\UnaryInfC{$p \wedge r \vdash_\Delta q \wedge r$}
				\RightLabel{$\eqref{rule:conjunction}$}
				\UnaryInfC{$p \wedge r \vdash_\Delta q$}
			\end{prooftree}
			\item By axiom $\eqref{rule:implication}$, it suffices to show $(q \Rightarrow \bot) \wedge (q \wedge p) \vdash_\Delta \bot$, for which it suffices to show $((q \Rightarrow \bot) \wedge q) \wedge p) \vdash_\Delta \bot$. By part 2 of this Lemma, it suffices to show $(q \Rightarrow \bot) \wedge q \vdash \bot$, which follows from axiom $\eqref{rule:implication}$, as $q \Rightarrow \bot \vdash_\Delta q \Rightarrow \bot$.
			\item By $\eqref{rule:implication}$, it suffices to show $q \vee r \vdash_\Delta p \Rightarrow ((p \wedge q) \vee (p \wedge r))$, for which by $\eqref{rule:disjunction}$, it suffices to show both $q \vdash_\Delta p \Rightarrow ((p \wedge q) \vee (p \wedge r))$, and $r \vdash_\Delta p \Rightarrow ((p \wedge q) \vee (p \wedge r))$. The first sequent can be proved by the following, 
			where the label $(1)$ is referring to the first part of this Lemma.
			\begin{prooftree}
				\AxiomC{$(p \wedge q) \vee (p \wedge r) \vdash_\Delta (p \wedge q) \vee (p \wedge q)$}
				\RightLabel{$\eqref{rule:disjunction}$}
				\UnaryInfC{$p \wedge q \vdash_\Delta (p \wedge q) \vee (p \wedge q)$}
				\RightLabel{$(1)$}
				\UnaryInfC{$q \wedge p \vdash_\Delta (p \wedge q) \vee (p \wedge q)$}
				\RightLabel{$\eqref{rule:implication}$}
				\UnaryInfC{$q \vdash_\Delta p \Rightarrow ((p \wedge q) \vee (p \wedge r))$}
			\end{prooftree}
			There is a similar proof tree for the remaining sequent.
			\item From part 4 of this Lemma, it suffices to show $(p \wedge \neg q) \vee (\neg q \wedge q) \vdash_\Delta p$, for which there is the following proof tree.
			\begin{scprooftree}{0.7}
				\AxiomC{$p \wedge (q \Rightarrow \bot) \vdash_{\Delta} p \wedge (q \Rightarrow \bot)$}
				\RightLabel{$\eqref{rule:conjunction}$}
				\UnaryInfC{$p \wedge (q \Rightarrow \bot) \vdash_{\Delta} p$}
				\AxiomC{$(q \Rightarrow \bot) \vdash_\Delta (q \Rightarrow \bot)$}
				\RightLabel{$\eqref{rule:implication}$}
				\UnaryInfC{$(q \Rightarrow \bot) \wedge q \vdash_\Delta \bot$}
				\AxiomC{}
				\RightLabel{$\eqref{rule:true_false}$}
				\UnaryInfC{$\bot \vdash_\Delta p$}
				\RightLabel{$\eqref{rule:cut}$}
				\BinaryInfC{$(q \Rightarrow \bot) \wedge q \vdash_\Delta p$}
				\RightLabel{$\eqref{rule:disjunction}$}
				\BinaryInfC{$(p \wedge (q \Rightarrow \bot)) \vee ((q \Rightarrow \bot) \wedge q) \vdash_\Delta p$}
			\end{scprooftree}
			\item Observe the following proof tree.
			\begin{prooftree}
				\AxiomC{$\exists x:\tau, p \wedge q \vdash_\Delta \exists x:\tau, p \wedge q$}
				\RightLabel{$\eqref{rule:existential}$}
				\UnaryInfC{$p \wedge q \vdash_{\Delta,x:\tau}\exists x:\tau, p \wedge q$}
				\RightLabel{$\eqref{rule:implication}$}
				\UnaryInfC{$p \vdash_{\Delta, x:\tau} q \Rightarrow (\exists x:\tau, p \wedge q)$}
				\RightLabel{$\eqref{rule:existential}$}
				\UnaryInfC{$\exists x:\tau, p \vdash_\Delta q \Rightarrow (\exists x : \tau, p \wedge q)$}
				\RightLabel{$\eqref{rule:implication}$}
				\UnaryInfC{$(\exists x: \tau, p) \wedge q \vdash_\Delta \exists x:\tau, p \wedge q$}
			\end{prooftree}
		\end{enumerate}
		Reading this same proof tree from bottom to top gives a proof tree for the second sequent.
	\end{proof}
	\begin{lemma}
		\label{lem:variablesub}
		If $x:\tau$ is a variable, and $t:\tau$ is a term, then
		\[x = t \vdash_{\Delta,x:\tau}s \qquad \text{if and only if}\qquad \vdash_{\Delta}s[x := t]\]
	\end{lemma}
	\begin{proof} The ``only if" direction follows from the following proof tree.
		\begin{prooftree}
			\AxiomC{}
			\RightLabel{$\eqref{rule:equality}$}
			\UnaryInfC{$\vdash_{\Delta}t = t$}
			\AxiomC{$\vdots$}
			\noLine
			\UnaryInfC{$x = t \vdash_{\Delta,x:\tau}s$}
			\RightLabel{$\eqref{rule:sub}$}
			\UnaryInfC{$t = t \vdash_{\Delta}s[x := t]$}
			\RightLabel{$\eqref{rule:cut}$}
			\BinaryInfC{$\vdash_{\Delta}s[x := t]$}
		\end{prooftree}
		For the other direction, let $x':\tau$ be such that $x' \not\in\text{FV}(s)$, then there is the following proof tree.
		\begin{prooftree}
			\AxiomC{$\vdots$}
			\noLine
			\UnaryInfC{$\vdash_{x,\Delta}s[x := t]$}
			\AxiomC{}
			\RightLabel{$\eqref{rule:equality}$}
			\UnaryInfC{$x = x' \wedge s[x := x']\vdash_{\Delta,x:\tau,x':\tau}s$}
			\RightLabel{$\eqref{rule:sub}$}
			\UnaryInfC{$x = t \wedge s[x := t]\vdash_{\Delta,x:\tau}s$}
			\RightLabel{$\eqref{weakening}$}
			\UnaryInfC{$s[x := t] \wedge x = t \vdash_{\Delta,x:\tau}s$}
			\RightLabel{$\eqref{rule:implication}$}
			\UnaryInfC{$s[x := t]\vdash_{\Delta,x:\tau}(x = t)\Rightarrow s$}
			\RightLabel{$\eqref{rule:cut}$}
			\BinaryInfC{$\vdash_{\Delta,x:\tau}(x = t)\Rightarrow s[x := t]$}
			\RightLabel{$\eqref{rule:implication}$}
			\UnaryInfC{$x = t\vdash_{\Delta,x:\tau}s[x:=t]$}
		\end{prooftree}
	\end{proof}
	\begin{lemma}
		\label{dansaside}
		If $p \vdash_{\Delta, x:\tau} q$, then $\exists x : \tau,p \vdash_\Delta \exists x : \tau, q$.
	\end{lemma}
	\begin{proof}
		First, let $\pi$ denote the following proof tree.
		\begin{prooftree}
			\AxiomC{$\forall x : \tau, q \vdash_\Delta \forall x : \tau, q$}
			\RightLabel{$\eqref{rule:universal}$}
			\UnaryInfC{$\forall x : \tau, q \vdash_{\Delta,x:\tau} q$}
			\AxiomC{$\exists x : \tau, q \vdash_\Delta \exists x : \tau, q$}
			\RightLabel{$\eqref{rule:existential}$}
			\UnaryInfC{$q \vdash_{\Delta, x:\tau} \exists x : \tau, q$}
			\RightLabel{$\eqref{rule:cut}$}
			\BinaryInfC{$\forall x : \tau, q \vdash_{\Delta,x:\tau}\exists x : \tau, q$}
		\end{prooftree}
		Then there is the following.
		\begin{prooftree}
			\AxiomC{$\vdots$}
			\noLine
			\UnaryInfC{$p \vdash_{\Delta,x:\tau} q$}
			\RightLabel{$\eqref{rule:universal}$}
			\UnaryInfC{$p \vdash_\Delta \forall x : \tau, q$}
			\RightLabel{$\eqref{rule:sub}$}
			\UnaryInfC{$p \vdash_{\Delta, x:\tau} \forall x : \tau, q$}
			\AxiomC{$\pi$}
			\noLine
			\UnaryInfC{$\vdots$}
			\noLine
			\UnaryInfC{$\forall x : \tau, q \vdash_{\Delta, x:\tau}\exists x : \tau, q$}
			\RightLabel{$\eqref{rule:cut}$}
			\BinaryInfC{$p \vdash_{\Delta, x:\tau} \exists x : \tau, q$}
			\RightLabel{$\eqref{rule:existential}$}
			\UnaryInfC{$\exists x : \tau, p \vdash_\Delta \exists x : \tau, q$}
		\end{prooftree}
	\end{proof}
	\begin{lemma}
		\label{lem:witness}
		Let $t:\tau$ be a term and $p$ a formula. Then
		\[p[x := t] \vdash_{\Delta} \exists x : \tau, p\]
	\end{lemma}
	The term $t$ can be thought of as a witness of the statement $p$, so this Lemma states that to entail $\exists x : \tau, p$, it suffices to bear a witness $t$.
	\begin{proof} Observe the following proof tree,
		\begin{prooftree}
			\AxiomC{$\exists x : \tau, p \vdash_{\Delta}\exists x : \tau, p$}
			\RightLabel{$\eqref{rule:existential}$}
			\UnaryInfC{$p \vdash_{\Delta, x:\tau}\exists x : \tau, p$}
			\RightLabel{$\eqref{rule:sub}$}
			\UnaryInfC{$p[x := t]\vdash_{\Delta} \exists x:\tau, p$}
		\end{prooftree}
	\end{proof}
	
	\begin{lemma}\label{lem:disjunction_right}
		If $p \vdash q$ then $p \vdash q \vee r$.
	\end{lemma}
	\begin{proof}
		Observe the following proof tree.
		\begin{prooftree}
			\AxiomC{$p \vdash q$}
			\AxiomC{$q \vee r \vdash q \vee r$}
			\RightLabel{\eqref{rule:disjunction}}
			\UnaryInfC{$q \vdash q \vee r$}
			\RightLabel{\eqref{rule:cut}}
			\BinaryInfC{$p \vdash q \vee r$}
		\end{prooftree}
	\end{proof}
	
	\subsection{Internal logic}
	\label{MitchellBenabou}
	%
	%
	%
	%
	%
	%
	As already mentioned, if $\call{E}$ is an elementary topos, then there is an associated type theory in the sense of Definition \ref{def:type_theory}, called the \emph{pure internal logic of $\call{E}$}. We now define this.
	\begin{definition}
		\label{def:internal_logic_pure}
		Let $\call{E}$ be a topos which admits countably infinite colimits. Choose a colimit for every cocone, and a limit for every finite cone. Also choose an initial object $\mathbbm{1}$, a subobject classifer $\Omega$ along with a family of exponentials for $\Omega$. The \textbf{pure internal logic of $\call{E}$} (or \textbf{pure internal language}) is a type theory in the sense of Definition \ref{def:type_theory}, which for every object $A \in \call{E}$ admits a type with the same name. If $C \in \call{E}$ is the chosen product of objects $A$ and $B$, then identify the type $C$ with the product type $A \times B$. Similarly, if $C$ is the chosen exponent $\Omega^A$ for some obejct $A \in \call{E}$, then identify the type $PA$ with the type $C$. The function symbols are all $f:A \to B$, where $f$ is a morphism in $\call{E}$, the relation symbols are all $r$ where $r: R \rightarrowtail A$ is a subobject of $A$ in $\call{E}$, and terms and formulas consisting only of those which can be constructed from the formation rules as given in Definition \ref{def:type_theory}.
	\end{definition}

\begin{remark}
	It is important in Definition \ref{def:internal_logic_pure} that the relation symbols are subobjects and \emph{not} chosen representatives of subobjects. The reason is due to the Substitution Lemma (\ref{lem:sub_prop} below). We show an example of what can go wrong if particular representatives are taken. Let $r$ be a formula with free variable $x$ of type $X$, and $t_1:X,t_2:Y$ terms, for simplicity we assume that $t_1$ has a single free variable $y:Y$ and $t_2$ is closed. We assume also that $y:Y$ does not appear inside $r$. Since the terms $(r[x:= t_1])[y := t_2]$ and $r[x := t_1[y := t_2]]$ are equal to each other, in order for the substitution Lemma to be well defined, we need the following to hold.
	\begin{equation}\label{eq:homotopy}
	\llbracket . (r[x:= t_1])[y := t_2] \rrbracket = \llbracket .  r[x := t_1[y := t_2]]\rrbracket
	\end{equation}
Now, $\llbracket. (r[x:= t_1])[y := t_2] \rrbracket$ is given by the following pair of pullback squares.
\begin{equation}\label{eq:composite}
	\begin{tikzcd}
		\llbracket . (r[x:= t_1])[y := t_2] \rrbracket\arrow[d]\arrow[r] &  \llbracket y: Y. r[x := t_1] \rrbracket\arrow[d]\arrow[r] & \llbracket x:X. r \rrbracket\arrow[d,"r"]\\
		\mathbbm{1}\arrow[r,"{\llbracket . t_2 \rrbracket}"] & Y\arrow[r,"{\llbracket y:Y. t_1\rrbracket}"] & X
	\end{tikzcd}
\end{equation}
On the other hand, $\llbracket .r [x:= t_1[y:= t_2]]\rrbracket$ is such that the following is a pullback diagram.
\begin{equation}\label{eq:subst}
	\begin{tikzcd}
		\llbracket . r[x := t_1[y := t_2]]\rrbracket\arrow[d]\arrow[r] & \llbracket x:X. r\rrbracket\arrow[d,"{r}"]\\
		\mathbbm{1}\arrow[r,"{\llbracket t_1[y := t_2]\rrbracket}"] & X
	\end{tikzcd}
\end{equation}
Hence, if we chose a subobject representative for each subobject, equality \eqref{eq:homotopy} would not be guaranteed. However, these two subobjects are necessarily isomorphic by virtue of them both being pullbacks of the same Diagram, as the morphisms in the bottom row of \eqref{eq:composite}, \eqref{eq:subst} are equal. Thus, formulas are subobjects, not chosen representatives of subobjects.
\end{remark}

	The Definition of entailment for this type theory is delayed until after Definition \ref{def:interpretation} below.
	Largely following Johnstone \cite[\S D4.1]{Johnstone}, associated to the pure internal logic of a topos $\call{E}$ is an \emph{interpretation} of this type theory in $\call{E}$, ie, an assignment of an object in $\call{E}$ to every type, a morphism with codomain $B$ to every tuple $(\Delta,t:B)$ consisting of a term $t$ along with a suitable context (where, as was given in Definition \ref{def:type_theory}, a context $\Delta$  is a suitable context for $t$ if every free variable of $t$ appears in $\Delta$), and a monomorphism to every tuple $(\Delta, p)$ where $p$ is a formula and $\Delta$ is a suitable context. In more detail, if $\Delta = (x_1:A_1,...,x_n:A_n)$ is a context and we have assigned objects $\llbracket A_i\rrbracket$ to each type $A_i$, and $t:B$ is a term for which $\Delta$ is a suitable context, then $\llbracket \Delta . t \rrbracket$ is a morphism.
	\[
	\begin{tikzcd}[column sep = huge]
		\llbracket A_1 \rrbracket \times ... \times \llbracket A_n \rrbracket\arrow[r,"{\llbracket \Delta . t\rrbracket}"] & \llbracket B \rrbracket
	\end{tikzcd}
	\]
	If $p$ is a formula for which $\Delta$ is a suitable context, then $\llbracket \Delta . p \rrbracket$ is a subobject. We will use $\llbracket \Delta. p \rrbracket$ to denote both this morphism and the domain of this morphism.
	\begin{equation}
		\llbracket \Delta. p \rrbracket \rightarrowtail \llbracket A_1 \rrbracket \times \hdots \times \llbracket A_n \rrbracket
	\end{equation}
	In what follows, we will make use of the Heyting algebra structure on the category of subobjects $\operatorname{C}$ associated to an object $C \in \call{E}$ and the perspective on quantifiers as adjoint functors between such categories. For a review of this material, see the Appendix Section \ref{sec:topos_struc}. Also, we introduce some notation.
	\begin{notation}\label{not:adj}
		Recall that any topos $\call{E}$ admits all exponentials. That is, for triple $A,B,C$ of objects in $\call{E}$ there is an object $C^A$ and a natural bijection $\operatorname{Hom}(A \times B, C) \cong \operatorname{Hom}(B,C^A)$. We denote the image of $f:A \times B \lto C$ under this bijection by $\adj{f}_C: B \lto C^A$. In the setting where there is low risk of ambiguity we simply write $\adj{f}$ for $\adj{f}_C$.
	\end{notation}
	\begin{definition}
		\label{def:interpretation}
		The \textbf{interpretation} $\llbracket \cdot \rrbracket$, of a type, term or formula of the internal logic of $\call{E}$ is defined by induction on the collection of objects and morphisms.
		\begin{itemize}
			\item \textbf{Types:} If $A$ is an object of $\call{E}$ viewed as a type, then $\llbracket A \rrbracket$ is this object. If $A \times B$ is a product type, then let $\llbracket A \times B \rrbracket$ be the chosen product $\llbracket A \rrbracket \times \llbracket B \rrbracket$. For terms of the form $PA$, let $\llbracket PA \rrbracket$ be the object $\Omega^A$.
			\item \textbf{Terms and formulas:} assume in what follows that $\Delta$ is always a suitable context. Also, for any context $\Sigma = (x_1:A_1,...,x_n:A_n)$, let $\overline{\Sigma}$ be $\llbracket A_1\rrbracket \times ... \times \llbracket A_n\rrbracket$ in $\call{E}$.
			\begin{itemize}
				\item If $x:A$ is a variable, then 
				$$\llbracket \Delta . x\rrbracket = \pi_A: \overline{\Delta} \to \llbracket A\rrbracket$$
				is the projection onto $\llbracket A\rrbracket$. In the case where $\Delta = A$ then $\llbracket \Delta . x\rrbracket$ is the identity morphism $\operatorname{id}_A$.
				\item $\llbracket \Delta . \ast \rrbracket = \overline{\Delta} \to \llbracket\mathbbm{1}\rrbracket$, the unique morphism into the terminal object.
				\item Recall from Theorem \ref{thm:heyting_sub} that $\text{Sub}(\overline{\Delta})$ is a Heyting algebra. In light of this, let $\llbracket \Delta . \top \rrbracket$ be the terminal object of this category, and $\llbracket \Delta \mid \bot \rrbracket$ the initial object of this category.
				\item If $t: A$ and $s: B$, then we interpret the product term $\langle t, s \rangle$ as the product morphism.
				$$\llbracket \Delta . \langle t,s \rangle \rrbracket = ( \llbracket \Delta . t \rrbracket, \llbracket \Delta . s \rrbracket )$$
				\item If $t: A \times B, then $ $\llbracket \Delta . \operatorname{fst}(t)\rrbracket = \pi_L\llbracket\Delta . t \rrbracket$, similarly for $\operatorname{snd}(t)$.
				\item If $f:A \to B$ is a function symbol, and $t:A$ a term, then $\llbracket \Delta \mid ft \rrbracket$ is the composite
				\[
				\begin{tikzcd}[column sep = huge]
					\overline{\Delta} \arrow[r,"{\llbracket \Delta . t \rrbracket}"] & \llbracket A\rrbracket\arrow[r,"f"] & \llbracket B\rrbracket
				\end{tikzcd}
				\]
				\item If $R \subseteq A$ is a relation symbol corresponding to a monomorphism $r: R \to A$, and $t: A$ a term, then $\llbracket \Delta . R(t) \rrbracket$ is the subobject $B \rightarrowtail \overline{\Delta}$ such that the following is a pullback diagram.
				\[
				\begin{tikzcd}
					B\arrow[r]\arrow[d,rightarrowtail] & R\arrow[d,rightarrowtail,"r"]\\
					\overline{\Delta}\arrow[r,swap,"{\llbracket \Delta . t\rrbracket}"] & A
				\end{tikzcd}
				\]
				\item $\llbracket \Delta . t = s\rrbracket$ is the following subobject.
				\[\operatorname{Equaliser}(\llbracket \Delta . t\rrbracket, \llbracket \Delta . s \rrbracket) \rightarrowtail \overline{\Delta}\]
				\item In the notation of Definition \ref{heytingnotation}, if $p$ and $q$ are formulas, then we define the following.
				\begin{align*}
					\llbracket \Delta . p \wedge q \rrbracket &= \llbracket \Delta . p \rrbracket \wedge \llbracket \Delta . q \rrbracket,\\
					\llbracket \Delta . p \vee q \rrbracket &= \llbracket \Delta . p \rrbracket \vee \llbracket \Delta . q \rrbracket,\\
					\llbracket \Delta . p \Rightarrow q \rrbracket &= \llbracket \Delta . p \rrbracket \Rightarrow \llbracket \Delta . q \rrbracket.
				\end{align*}
				\item Similarly, if $\lbrace p_i\rbrace_{i = 0}^\infty$ is a countably infinite set of formulas, then
				\[\llbracket \Delta . \bigvee_{i = 0}^\infty p_i\rrbracket = \bigvee_{i=1}^\infty\llbracket \Delta . p_i\rrbracket\]
				where $\displaystyle\bigvee_{i=1}^{\infty}$ is the countably infinite coproduct of the category $\text{Sub}(\overline{\Delta})$, this exists as we have assumed that $\call{E}$ admits countably infinite colimits, see Remark \ref{rmk:countable_colimits}.
				\item As was done at the beginning of Section \ref{helpfulLemmas}, let $\Delta, a:A$ mean the context $\Delta$ with $a:A$ appended to the end. Also, $\Delta\setminus a:A$ will mean $\Delta$ with $a:A$ omitted (assuming $a:A$ appears in $\Delta$). $\llbracket \Delta . \exists a: A, p\rrbracket$ is the image (see Definition \ref{def:image}) of the following composite.
				\[
				\begin{tikzcd}[column sep = huge]
					\llbracket \Delta. p\rrbracket\arrow[r,rightarrowtail] & \overline{\Delta}\arrow[r,"{\pi}"] & \overline{\Delta\setminus a:A}
				\end{tikzcd}
				\]
				Here, the morphism $\pi$ is given by a product of projection morphisms.
				\item Recall from Theorem \ref{adj} that the functor $\pi^\ast: \text{Sub}(\overline{\Delta \setminus a:A}) \to \text{Sub}(\overline{\Delta})$ admits a right adjoint, in accordance with Definition \ref{qunatifiers}, we denote this adjoint by 
				$\forall_\pi$. We define the following.
				\begin{equation}
					\llbracket \Delta . \forall a:A, p\rrbracket = \forall_\pi(\llbracket \Delta,a:A . p \rrbracket)
				\end{equation}        
				\item $\llbracket \Delta . \lbrace x : A . p \rbrace \rrbracket$ is the transpose of the morphism $\chi_{\llbracket \Delta, a:A . p(t) \rrbracket}$ which in turn is the unique morphism such that the following is a pullback diagram.
				\[
				\begin{tikzcd}[column sep = huge]
					\llbracket \Delta, a:A . p \rrbracket\arrow[r]\arrow[d] & \mathbbm{1}\arrow[d,"{\text{true}}"]\\
					\overline{\Delta} \times A\arrow[r,swap,"{\chi_{\llbracket \Delta, a:A . p(t) \rrbracket}}"] & \Omega
				\end{tikzcd}
				\]
				That is, $\llbracket \Delta . \{ x:A \mid p \} \rrbracket: \overline{\Delta} \to \Omega^A$ is $\adj{\chi_{\llbracket \Delta, a:A . p\rrbracket}}$ (see Notation \ref{not:adj}).
				\item $\llbracket \Delta . t \in \lbrace x:A \mid p \rbrace\rrbracket$ is the monic $C \rightarrowtail \overline{\Delta}$ such that the following is a pullback diagram, in the following the morphism $\in_A \rightarrowtail A \times \Omega^A$ is the subobject corresponding to the morphism $\operatorname{Eval_A}: A \times \Omega^A \lto \Omega$.
				\[
				\begin{tikzcd}[column sep = huge]
					C\arrow[d,rightarrowtail]\arrow[r] & \in_A\arrow[d,rightarrowtail]\\
					\overline{\Delta}\arrow[r,swap] & A \times \Omega^A
				\end{tikzcd}
				\]
				where the bottom morphism is ${( \llbracket \Delta . t \rrbracket,\llbracket \Delta . \lbrace x:A . p \rbrace \rrbracket)}$.
			\end{itemize}
		\end{itemize}
	\end{definition}
	Now, Definition \ref{def:internal_logic_pure} can be completed
	\begin{definition}[Finishing Definition \ref{def:internal_logic_pure}]
		We define $p \vdash_{\Delta} q$ to mean $\llbracket \Delta . p \rrbracket \leq \llbracket \Delta . q \rrbracket$, where $\leq$ is the preorder on the set $\text{Sub}(\overline{\Delta})$.
	\end{definition}
	In Definition \ref{def:type_theory} we defined the notion of a type theory with no reference to a topos $\call{E}$. In Definition \ref{def:internal_logic_pure} we defined a type theory corresponding to a given topos $\call{E}$ (where $\call{E}$ admits countably infinite colimits) which we called the \emph{pure} internal logic associated to $\call{E}$. The deduction rules of the pure internal logical are rather restrictive, so we make a new definition (Definition \ref{def:internal_logic_arb}) which allows these deduction rules to be extended to include extra rules, the \emph{derived rules} (Definition \ref{def:sound_deduction_rules}), provided that the result remains sound, that is, what is provable in the internal logic reflects true statements when interpreted in $\call{E}$. What extra deduction rules are allowed is made precise in Definition \ref{def:sound_deduction_rules} and some examples are given in Example \ref{ex:surjective_pairing} and Proposition \ref{prop:derived_rules}.
	\begin{definition}\label{def:sound_deduction_rules}
		Consider the pure internal logic of a topos $\call{E}$ admitting all countably infinite colimits. A \textbf{derived rule} consists of a pair $(\lbrace \mathfrak{a}_i\rbrace_{i \in I}, \mathfrak{b})$ consisting of a set $\lbrace \mathfrak{a}_i\rbrace_{i\in I}$ of sequents in the pure internal logic along with a sequent $\mathfrak{b}$ also in the pure internal logic subject to the condition that the interpretation of $\mathfrak{b}$ holds whenever the interpretation of all $\mathfrak{a_i}$ hold.
		
		Derived rules $(\lbrace \mathfrak{a}_i\rbrace_{i \in I},\mathfrak{b})$ where the set $I$ is finite will be represented by a horizontal line with the sequents in $\lbrace \mathfrak{a}_i\rbrace_{i \in I}$ displayed above the line and the sequent $\mathfrak{b}$ below the line, similar to the deduction rules of Definition \ref{def:internal_logic_pure}. For example, the following notate, from left to right, a derived rule where the set $I$ has $0,1$ and $2$ elements respectively.
		\begin{center}
			\AxiomC{$p \vdash_\Delta q$}
			\DisplayProof
			\qquad\qquad
			\AxiomC{$p\vdash_\Delta q$}
			\UnaryInfC{$r\vdash_{\Sigma} s$}
			\DisplayProof
			\qquad\qquad
			\AxiomC{$p\vdash_\Delta q$}
			\AxiomC{$s \vdash_\Sigma r$}
			\BinaryInfC{$t\vdash_\Omega v$}
			\DisplayProof
		\end{center}
		Such a derived rule is \textbf{finite}.
		
		In the case where the set $I$ is infinite, we will denote a derived rule $(\lbrace \frak{a}_i\rbrace_{i\in I}, \frak{b})$ as follows. Such a deduction rule is \textbf{infinitary}.
		\begin{center}
			\AxiomC{$\lbrace \frak{a}_i\rbrace_{i \in I}$}
			\UnaryInfC{$\frak{b}$}
			\DisplayProof
		\end{center}
	A derived rule where the indexing set $I$ is empty is a \textbf{derived sequent}.
	\end{definition}
	\begin{example}\label{ex:surjective_pairing}
		A simple example of a sound deduction rule which is not a deduction rule of the pure internal logic is the following sequent.
		\begin{equation}
			\vdash_{z:A \times B} \exists a:A, \exists b:B, z = \langle a, b \rangle
		\end{equation}
		\begin{center}
			\AxiomC{$\vdash_{z:A \times B}z$}
		\end{center}
	\end{example}
	\begin{definition}\label{def:internal_logic_arb}
		An \textbf{extended internal logic} of a topos $\call{E}$ is given by the pure internal logic along with a finite set of derived rules. In this paper, we will only be concerned with one internal logic which is given by the pure internal logic along with the derived rules given in Proposition \ref{prop:derived_rules} below. This is the \textbf{internal logic} (or \textbf{internal language}) of $\call{E}$.
	\end{definition}
	According to Definition \ref{def:interpretation} terms in the pure internal logic are interpreted as morphisms, and formulas are interpreted as subobjects. How can the pure internal logic be used to describe morphisms between subobjects? One way is given by the following Lemma. We remark that this is stated without proof as Lemma 1.2.7 in \cite[\S D]{Johnstone}.
	\begin{lemma}[Substitution property]\label{lem:sub_prop}
		Let $b_1:B_1,\hdots,b_n:B_n$ be a sequence of variables and let $t_1:B_1,...,t_n:B_n$ be a sequence of terms with $t_i$ having the same type as $b_i$ for each $i = 1,...,n$. Consider also a formula $q$ and a context $\Delta$ suitable for all of $t_1,...,t_n$. There exists a morphism $\delta$ such that the following Diagram is a pullback square.
		\begin{equation}\adjustbox{scale=0.85}{
			\begin{tikzcd}[column sep = huge]
				\overline{\Delta}\arrow[rr,"{(\llbracket \Delta. t_1 \rrbracket,...,\llbracket \Delta. t_n \rrbracket)}"] && B_1 \times \hdots \times B_n\\
				\llbracket \Delta. q[(b_1,...,b_n) := (t_1,...,t_n)]\rrbracket\arrow[u,rightarrowtail]\arrow[rr,"{\delta}"] && \llbracket (b_1:B_1,...,b_n:B_n). q\rrbracket\arrow[u,rightarrowtail]
			\end{tikzcd}}
		\end{equation}
	\end{lemma}
	The pure internal logic is insufficient for the result we prove in Section \ref{sec:colimits}. We now define the set of derived rules which will be required.
	\begin{proposition}\label{prop:derived_rules}
		The following are derived rules.
		\begin{enumerate}
			\item\label{rule:element} Let $a:A$ be a variable of type $A$ and $b:B$ a variable of type $B$. Let $t(a)$ be a term with free variable set $\lbrace a:A\rbrace$ and let $p,q$ be formulas and assume $b:B \in \operatorname{FV}(q)$. We have the following:
			\begin{center}
				\AxiomC{$p\vdash_\Delta t(a) \in \lbrace b:B \mid q\rbrace$}
				\RightLabel{$(\operatorname{Mem})_1$}
				\UnaryInfC{$p \vdash_\Delta q[b := t(a)]$}
				\DisplayProof
				\end{center}
			as well as:
			\begin{center}
				\AxiomC{$p \vdash_{\Delta}q[b := t(a)]$}
				\RightLabel{$(\operatorname{Mem})_2$}
				\UnaryInfC{$p\vdash_{\Delta}t(a) \in \lbrace b:B \mid q\rbrace$}
				\DisplayProof
			\end{center}
			(The label $\operatorname{Mem}$ stands for ``member")
			\item\label{rule:infinite} Consider a countable set of formulas $\lbrace p_i\rbrace_{i \geq 0}$, a formula $q$, a context $\Sigma$ such that the free variables of all $p_i$ and $q$ appear in $\Sigma$, and a countable set of sequents $p_i \vdash_{\Sigma} q$. Recall also the notation for infinite derived rules given in Definition \ref{def:sound_deduction_rules}.
			\begin{center}
				\AxiomC{$\lbrace p_i \vdash_{\Sigma} q\rbrace_{i \geq 0}$}
				\RightLabel{$(\bigvee I)$}
				\UnaryInfC{$\bigvee_{i \geq 0}p_i \vdash_\Sigma q$}
				\DisplayProof
				\qquad
				\AxiomC{$\bigvee_{i\geq 0}p_i \vdash_\Sigma q$}
				\RightLabel{$(\bigvee E)_i$}
				\UnaryInfC{$p_i \vdash_{\Sigma} q$}
				\DisplayProof
			\end{center}
			\item\label{rule:set} Let $p,q$ be formulas and $\Delta$ a valid context for them. We have the following:
			\begin{center}
				\AxiomC{$\vdash_{\Delta} \lbrace a:A \mid p\rbrace = \lbrace a:A \mid q \rbrace$}
				\RightLabel{$(\lbrace \cdot \rbrace E)$}
				\UnaryInfC{$\vdash_{\Delta,a:A}p \Leftrightarrow q$}
				\DisplayProof
				\end{center}
			as well as:
			\begin{center}
				\AxiomC{$\vdash_{\Delta,a:A}p \Leftrightarrow q$}
				\RightLabel{$(\lbrace \cdot \rbrace I)$}
				\UnaryInfC{$\vdash_{\Delta} \lbrace a:A \mid p\rbrace = \lbrace a:A \mid q \rbrace$}
				\DisplayProof
			\end{center}
		\end{enumerate}
	\end{proposition}
	\begin{proof}
		First we prove \eqref{rule:element}.
		
		This is mostly a matter of unwinding definitions. First we calculate $\llbracket \Delta. t(a) \in \lbrace b:B \mid q\rbrace\rrbracket$. By Definition \ref{def:interpretation} we have that $\llbracket \Delta. \lbrace b:B \mid q\rbrace\rrbracket$ is the transpose of the unique morphism $\chi_{\llbracket \Delta,b:B. q\rrbracket}$ such that the following is a pullback diagram.
		\begin{equation}
			\begin{tikzcd}[column sep = huge]
				\llbracket \Delta,b:B. q\rrbracket\arrow[r]\arrow[d,rightarrowtail] & \mathbbm{1}\arrow[d]\\
				\overline{\Delta} \times B\arrow[r,swap,"{\chi_{\llbracket \Delta, b:B. q\rrbracket}}"] & \Omega
			\end{tikzcd}
		\end{equation}
		Then, the morphism $\llbracket \Delta. t(a) \in \lbrace b:B \mid q\rbrace \rrbracket$ is the monomorphism $\overline{\Delta} \lto B \times \Omega^B$ which has been chosen so that the following is a pullback square.
		\begin{equation}
			\begin{tikzcd}[column sep = huge]
				\llbracket \Delta. t(a) \in \lbrace b:B \mid q \rbrace\rrbracket\arrow[rr]\arrow[d,rightarrowtail] & & \in_B\arrow[d,rightarrowtail]\\
				\overline{\Delta}\arrow[rr,swap,"{(\llbracket \Delta. t(a)\rrbracket, \adj{\chi_{\llbracket \Delta,b:B. q\rrbracket}})}"] & & B \times \Omega^B
			\end{tikzcd}
		\end{equation}
		Next we describe $\llbracket \Sigma. q[b:= t(a)]$. We will need an explicit description for $\Sigma$, say $\Sigma = (b:B, b_1:B_1,...,b_n:B_n)$, there is a small loss of generality here because we assume that $b:B$ appears as the first element in $\Sigma$ but the proof of the other cases differ only trivially to this case. By Lemma \ref{lem:sub_prop} we have that the following is a pullback Diagram, where $b_1:B_1,...,b_n:B_n$ are variables.
		\begin{equation}
			\begin{tikzcd}[column sep = huge]
				\llbracket \Delta. q[b:=t(a)]\rrbracket\arrow[rr]\arrow[d] && \llbracket\Sigma. q \rrbracket \arrow[d,rightarrowtail]\\
				\overline{\Delta}\arrow[rr,swap,"{(\llbracket \Delta. t(a)\rrbracket,\llbracket \Delta. b_1\rrbracket,\hdots, \llbracket \Delta. b_n\rrbracket)}"] && B \times B_1 \times \hdots B_n
			\end{tikzcd}
		\end{equation}
		Now, to show that $(\operatorname{Mem})_1,(\operatorname{Mem}_2)$ are derived rules, we have to show the following.
		\begin{equation}
			\llbracket \Delta. p\rrbracket \leq \llbracket \Delta. t(a) \in \lbrace b:B \mid q\rbrace \rrbracket\text{ if and only if }\llbracket \Delta. p\rrbracket \leq \llbracket \Delta. q[b:=t(a)]\rrbracket
		\end{equation}
		In other words, we must show that there exists a morphism $c: \llbracket \Delta. p\rrbracket \lto \llbracket \Delta. t(a) \in \lbrace b:B \mid q \rbrace \rrbracket$ such that the Diagram on the left of \eqref{eq:entails_interp} commutes if and only if there exists a morphism $d: \llbracket \Delta. p(a)\rrbracket \lto \llbracket \Delta. q[b:=t(a)]\rrbracket$ so that the Diagram on the right of \eqref{eq:entails_interp} commutes. The proof of both directions are similar to each other so we only show the forwards implication, we have labelled the morphism $\llbracket \Delta. p(a)\rrbracket \rightarrowtail \overline{\Delta}$ by $\delta$ for the sake of future reference.
		\begin{equation}\label{eq:entails_interp}
			\begin{tikzcd}
				\llbracket \Delta. p(a)\rrbracket\arrow[r,"c"]\arrow[dr,rightarrowtail,swap,"{\delta}"] & \llbracket \Delta. t(a) \in \lbrace b:B \mid q \rbrace \rrbracket\arrow[d,rightarrowtail]\\
				& \overline{\Delta}
			\end{tikzcd}
			\qquad
			\begin{tikzcd}
				\llbracket \Delta. p(a)\rrbracket\arrow[r,"d"]\arrow[dr,rightarrowtail, swap, "{\delta}"] & \llbracket \Delta. q[b:=t(a)]\rrbracket\arrow[d,rightarrowtail]\\
				& \overline{\Delta}
			\end{tikzcd}
		\end{equation}
		First, we consider Diagram \eqref{eq:intermediate_commuting}. The bottom right triangle of this Diagram commutes as $\Omega^B$ is an exponential. The top left triangle commutes just by definition of the morphisms involved. Hence the outer square is commutative.
		\begin{equation}
			\begin{tikzcd}[column sep = huge, row sep = huge]\label{eq:intermediate_commuting}
				\overline{\Delta}\arrow[rrr,"{(\llbracket \Delta. t(a)\rrbracket, \adj{\chi_{\llbracket \Delta, b:B. q\rrbracket}})}"]\arrow[dd,swap,"{(\llbracket \Delta. t(a)\rrbracket,\llbracket \Delta. b_1\rrbracket,...,\llbracket \Delta. b_n\rrbracket)}"] & & & B \times \Omega^B \arrow[dd,"{\operatorname{Eval}_B}"]\\
				\\
				\overline{\Sigma}\arrow[rrr,swap,"{\chi_{\llbracket \Sigma. q\rrbracket}}"]\arrow[uurrr,"{(\pi_B, \adj{\chi_{\llbracket \Sigma. q\rrbracket}}\pi_{\overline{\Sigma\setminus B}})}"] & & & \Omega
			\end{tikzcd}
		\end{equation}
		We now use this to prove the statement at hand. Say the left Triangle of \eqref{eq:entails_interp} commutes.  We consider the following two Diagrams.
		\begin{equation}\label{eq:pullback_squares}
			\begin{tikzcd}[column sep = huge]
				C\arrow[rr]\arrow[d] & & \in_B\arrow[r]\arrow[d,rightarrowtail] & \mathbbm{1}\arrow[d,swap,"{\operatorname{True}}"]\\
				\overline{\Delta}\arrow[rr,swap,"{(\llbracket \Delta. t(a)\rrbracket, \adj{\chi_{\llbracket \Delta, b:B. q\rrbracket}})}"] & & B \times \Omega^B\arrow[r,swap,"{\operatorname{Eval}_B}"] & \Omega
			\end{tikzcd}
			\end{equation}
		\begin{equation}
			\begin{tikzcd}[column sep = huge]
				D\arrow[rr]\arrow[d] &&\llbracket \Sigma. q\rrbracket\arrow[r]\arrow[d,rightarrowtail,"q"] & \mathbbm{1}\arrow[d,"{\operatorname{True}}"]\\
				\overline{\Delta}\arrow[rr,swap,"{(\llbracket \Delta. t(a)\rrbracket,\llbracket \Delta,.b_1\rrbracket,...,\llbracket \Delta. b_n\rrbracket)}"] && \overline{\Sigma}\arrow[r,swap,"{\chi_{\llbracket \Sigma. q\rrbracket}}"] & \Omega
			\end{tikzcd}
		\end{equation}
		Both of these Diagrams consist of two pullback squares, and so both outer squares are pullback squares. Since the outer square of the bottom Diagram of \eqref{eq:pullback_squares} is a pullback square, to construct a morphism $\llbracket \Delta. p(a)\rrbracket \lto D$ it suffices to define a morphism $\gamma: \llbracket \Delta. p(a)\rrbracket \lto \overline{\Delta}$ such that the following composite:
		\begin{equation}
			\begin{tikzcd}[column sep = huge]
				\llbracket \Delta. p\rrbracket\arrow[r,rightarrowtail,"{\gamma}"] & \overline{\Delta}\arrow[rrr,"{(\llbracket \Delta. t(a)\rrbracket,\llbracket \Delta. b_1\rrbracket,\hdots, \llbracket \Delta. b_n\rrbracket)}"] &&& \overline{\Sigma}\arrow[d,"{\chi_{\llbracket b:B. q\rrbracket}}"]\\
				&&&& \Omega
			\end{tikzcd}
		\end{equation}
		is equal to the following composite.
		\begin{equation}\label{eq:true_stuff}
			\begin{tikzcd}
				\llbracket \Delta. p\rrbracket\arrow[r] & \mathbbm{1}\arrow[r,"{\operatorname{True}}"] & \Omega
			\end{tikzcd}
		\end{equation}
		By commutativity of the outer square on the top Diagram of \eqref{eq:pullback_squares}, we have that the morphism \eqref{eq:true_stuff} is equal to the following composite.
		\begin{equation}
			\begin{tikzcd}[column sep = huge]
				\llbracket \Delta. p(a)\rrbracket\arrow[r,rightarrowtail,"{\delta}"] & \overline{\Delta}\arrow[rr,"{(\llbracket \Delta. t(a)\rrbracket, \adj{\chi_{\llbracket \Delta, b:B. q\rrbracket}})}"] & & B \times \Omega^B\arrow[d,"{\operatorname{Eval}_B}"]\\
				&&& \Omega
			\end{tikzcd}
		\end{equation}
		Hence, by observing commutativity of Diagram \eqref{eq:intermediate_commuting}, we see that $\delta$ is an appropriate choice for $\gamma$.
		
		Now we prove \eqref{rule:infinite}. We notice that the collection of subobjects $\llbracket \Sigma. p_i\rrbracket$ factors through the subobject $\llbracket \Sigma. q\rrbracket$ if and only if the coproduct $\bigvee_{i \geq 0}\llbracket \Sigma. p_i \rrbracket$ (in the category $\operatorname{Sub}(\overline{\Sigma})$) factors through $\llbracket \Sigma. q \rrbracket$. This statement is exactly the statement that $(\bigvee I)$ and $(\bigvee E)_i$ in \eqref{rule:infinite} are derived rules.
		
		Finally, we prove \eqref{rule:set}. In light of the derived deduction rules \eqref{rule:element} it suffices to show the following are derived deduction rules.
		\begin{center}
			\AxiomC{$\vdash_{\Delta} \lbrace a:A \mid p \rbrace = \lbrace a:A \mid q \rbrace$}
			\RightLabel{$(\operatorname{Leib})_1$}
			\UnaryInfC{$\vdash_{\Delta,a':A}a' \in \lbrace a:A \mid p \rbrace \Leftrightarrow a' \in \lbrace a:A \mid q \rbrace$}
			\DisplayProof
			\end{center}
		and
		\begin{center}
			\AxiomC{$\vdash_{\Delta,a':A}a' \in \lbrace a:A \mid p \rbrace \Leftrightarrow a' \in \lbrace a:A \mid q \rbrace$}
			\RightLabel{$(\operatorname{Leib})_2$}
			\UnaryInfC{$\vdash_{\Delta} \lbrace a:A \mid p \rbrace = \lbrace a:A \mid q \rbrace$}
			\DisplayProof
		\end{center}
		Recall from Definition \ref{def:interpretation} that $\llbracket \Delta. \lbrace a:A \mid p \rbrace = \lbrace a:A \mid q\rbrace\rrbracket \rightarrowtail \overline{\Delta}$ is the monomorphism rendering the following an equaliser diagram.
		\begin{equation}
			\begin{tikzcd}[column sep = large]
				\llbracket \Delta. \lbrace a:A \mid p \rbrace = \lbrace a:A \mid q\rbrace\rrbracket\arrow[r,rightarrowtail] & \overline{\Delta}\arrow[rr,shift left, "{\llbracket \Delta. \lbrace a:A \mid p \rbrace\rrbracket}"]\arrow[rr,shift right,swap,"{\llbracket \Delta. \lbrace a:A \mid q \rbrace \rrbracket}"] && \Omega^A
			\end{tikzcd}
		\end{equation}
		Hence, the sequent $\vdash_\Delta \lbrace a:A \mid p \rbrace = \lbrace a:A \mid q \rbrace$ holds if and only if $\llbracket \Delta. \lbrace a:A \mid p \rbrace\rrbracket = \llbracket \Delta. \lbrace a:A \mid q \rbrace\rrbracket$.
		
		On the other hand, $\llbracket \Delta, a':A. a' \in \lbrace a:A \mid p \rbrace \rrbracket$ and $\llbracket \Delta, a':A. a' \in \lbrace a:A \mid p \rbrace \rrbracket$ are such that the following Diagrams are pullback squares.
		\begin{equation}
			\begin{tikzcd}[column sep = large]
				\llbracket \Delta, a':A. a' \in \lbrace a:A \mid p \rbrace \rrbracket\arrow[d,rightarrowtail]\arrow[rr] && \in_A\arrow[d,rightarrowtail]\\
				\overline{\Delta, a:A}\arrow[rr,swap,"{(\llbracket \Delta, a:A.a\rrbracket, \llbracket \Delta,a:A. \lbrace a:A \mid p \rbrace\rrbracket)}"] && A \times \Omega^A
			\end{tikzcd}
			\end{equation}
		\begin{equation}
			\begin{tikzcd}[column sep = large]
				\llbracket \Delta, a':A. a' \in \lbrace a:A \mid q \rbrace \rrbracket\arrow[d,rightarrowtail]\arrow[rr] && \in_A\arrow[d,rightarrowtail]\\
				\overline{\Delta, a:A}\arrow[rr,swap,"{(\llbracket \Delta, a:A.a\rrbracket, \llbracket \Delta,a:A. \lbrace a:A \mid q \rbrace\rrbracket)}"] && A \times \Omega^A
			\end{tikzcd}
		\end{equation}
		Hence, that these two subobjects being equal to each other is equivalent to the statement that $\llbracket \Delta, a:A. \lbrace a :A \mid p \rbrace \rrbracket = \llbracket \Delta, a:A. \lbrace a:A \mid q \rbrace \rrbracket$. This shows that $(\operatorname{Leib})_1,(\operatorname{Leib})_2$ are derived deduction rules.
	\end{proof}
	
	\section{Using an internal logic}\label{sec:crucial_lemma}
	It is often convenient when working with an internal logic to have all objects and morphisms of current interest described via the internal logic. Therefore, we dedicate Section \ref{sec:object_term} to showing how to construct a description of an arbitrary element $U$ of the topos via a formula in an internal logic. After this, in Section \ref{sec:crucial_terms} we define some terms of particular interest to our aim of describing finite colimits using an internal logic.
	\subsection{The term associated to an object}\label{sec:object_term}
	Say we have a topos $\call{E}$ admitting all countably infinite colimits, and a morphism $f: W \lto U$ in $\call{E}$. By Definition \ref{def:internal_logic_pure} the morphism $f$ corresponds to a function symbol in the internal logic. Hence, given a variable $w:W$ of type $W$, we can construct the term $fw:U$. Naturally, one wishes to construct a term such as ``$fw \in U$" but we must notice that this is an invalid construction! We are only able to construct a term ``$fw \in s$" when $s$ is a term of type $\call{P}U$, hence ``$fw \in U$" is meaningless. However, in the topos $\underline{\operatorname{Set}}$ of sets, there is a natural bijection between an arbitrary set $U$ and the set $\lbrace z \in \call{P}U \mid \exists u \in U, z = \lbrace u \rbrace \rbrace$. This generalises immediately; if $U$ is an arbitrary object of a topos $\call{E}$ then the following term may be constructed $\lbrace z: \call{P}U \mid \exists u \in U, z = \lbrace u \rbrace \rbrace$. We prove in Corollary \ref{cor:embedded_subobject} that there is an isomorphism between an arbitrary object $U$ in $\call{E}$ and the domain of the morphism $\llbracket z: \call{P}U. \exists u:U, z = \lbrace u \rbrace \rrbracket$.  Hence, in lieu of being able to write $fa \in U$ we may write $\lbrace fa\rbrace \in \lbrace z: \call{P}U \mid \exists u : U, z = \lbrace u \rbrace \rbrace$.
	
	\begin{lemma}\label{lem:image_interpretation}
		Let $f: A \lto B$ be a morphism. Consider the following.
		\begin{equation}
			\llbracket b:B. \exists a \in A, fa = b\rrbracket\rightarrowtail B
		\end{equation}
		There exists a unique isomorphism $\psi:  \operatorname{im}f \lto\llbracket b:B. \exists a \in A, fa = b \rrbracket$, where $\operatorname{im}f$ is the image of $f$ (Definition \ref{def:image}) so that the following diagram commutes.
		\begin{equation}
			\begin{tikzcd}
				\operatorname{im}f\arrow[dr,rightarrowtail]\arrow[r,"{\psi}"] & \llbracket b:B. \exists a \in A, fa = b\rrbracket\arrow[d,rightarrowtail]\\
				& B
			\end{tikzcd}
		\end{equation}
	\end{lemma}
	\begin{proof}
		First we identify the morphism $\llbracket b:B. \exists a:A, fa = b\rrbracket$. The morphism $\pi_B: A \times B \lto B$ induces an adjunction $\exists_{\pi_B} \dashv\pi_B^{-1}$, and by Definition \ref{def:internal_logic_pure} we have
		\begin{equation}
			\llbracket b:B. \exists a:A, fa = b\rrbracket = \exists_{\pi_B} \big(\llbracket a:A, b:B. fa = b\rrbracket\big)
		\end{equation}
		The morphism $\llbracket a:A, b:B. fa = b\rrbracket$ is the Equaliser in the following Diagram.
		\begin{equation}
			\begin{tikzcd}
				\operatorname{Equal}(f\pi_A,\pi_B)\arrow[r,"{e}"] & A \times B\arrow[r,shift left,"{f\pi_A}"]\arrow[r,shift right,swap,"{\pi_B}"] & B
			\end{tikzcd}
		\end{equation}
		Hence, $\llbracket b:B. \exists a:A, fa = b\rrbracket$ is equal to the image $\operatorname{im}(\pi_Be)$ of $\pi_Be$. Our task is to prove that there exists an isomorphism $\operatorname{im}(f) \stackrel{\sim}{\lto} \operatorname{im}(\pi_Be)$. By the adjunction $\exists_{\pi_B} \dashv \pi_B^{-1}$, we have in particular the following natural bijection.
		\begin{equation}
			\operatorname{Hom}_{\operatorname{Sub}(B)}(\operatorname{im}\pi_B e, \operatorname{im}f) \cong \operatorname{Hom}_{\operatorname{Sub}(A \times B)}(e, \pi_B^{-1}(\operatorname{im}f))
		\end{equation}
		Hence it suffices to show the existence of an isomorphism $e \stackrel{\sim}{\lto} \pi_B^{-1}(\operatorname{im}(f))$. Denote the morphism $\operatorname{im}f \lto B$ by $i_1$ and the morphism $\pi_B^{-1}(\operatorname{im}f) \lto A \times B$ by $i_2$. We consider the following commutative diagram.
		\begin{equation}\label{eq:image_unwinding}
			\begin{tikzcd}
				& \pi_B^{-1}(\operatorname{im}f)\arrow[r,"{j}"]\arrow[d,swap,"{i_2}"] & \operatorname{im}f\arrow[d,"{i_1}"]\\
				\operatorname{Equal}(f\pi_A,\pi_B)\arrow[r,"e"] & A \times B\arrow[r,"{\pi_B}"] & B
			\end{tikzcd}
		\end{equation}
		Where the square on the right of Diagram \eqref{eq:image_unwinding} is a pullback Diagram. We will use the fact that this is a pullback Daigram to induce a morphism $\operatorname{Equal}(f\pi_A,\pi_B) \lto \pi_B^{-1}(\operatorname{im}f)$. To do this, we need to describe a morphism $d: \operatorname{Equal}(f\pi_A,\pi_B) \lto \operatorname{im}f$ so that $i_1 d = \pi_B e$.
		
		The morphism $\pi_B e$ is equal to the morphism $f\pi_Ae$, and $f$ in turn factors through $\operatorname{im}f$. Hence, there exists a morphism $\operatorname{Equal}(f\pi_A,\pi_B) \lto \operatorname{im}f \lto B$ which commutes with $\pi_B e$.
		
		Let $\gamma: \operatorname{Equal}(f\pi_A,\pi_B) \lto \pi_B^{-1}(\operatorname{im}f)$ denote the morphism induced by the universal property of the pullback.
		
		Next we define a morphism $\pi_B^{-1}(\operatorname{im}f) \lto \operatorname{Equal}(f\pi_A,\pi_B)$. We will use the universal property of the Equaliser.
		
		We simply need to show that $\pi_Bi_2 = f\pi_Ai_2$.  This follows from the following simple calculation.
		\begin{align*}
			\pi_B i_2 &= i_1 j\\
			&= f\pi_Ai_2
		\end{align*}
		Hence there exists a morphism $\delta: \pi_B^{-1}(\operatorname{im}f) \lto \operatorname{Equal}(f\pi_A,\pi_B)$.
		
		It remains to show that $\gamma \delta = \operatorname{id}$ and $\delta \gamma = \operatorname{id}$. This follows from the uniqueness of the induced morphisms coming from the universal properties of $\operatorname{Equal}(f\pi_A,\pi_B)$ and $\pi_B^{-1}(\operatorname{im}f)$.
	\end{proof}
	\begin{corollary}\label{cor:embedded_subobject}
		Let $u:U$ be a variable. Then there is a unique isomorphism $U \lto \llbracket z: \call{P}U. \exists u:U, z = \lbrace u \rbrace \rrbracket$ such that the following Diagram commutes.
		\begin{equation}
			\begin{tikzcd}
				U\arrow[r,"{\llbracket u:U. \lbrace u \rbrace \rrbracket}"]\arrow[dr] & \Omega^U\\
				& \llbracket z: \call{P}U . \exists u :U, z = \lbrace u \rbrace \rrbracket\arrow[u,rightarrowtail]
			\end{tikzcd}
		\end{equation}
	\end{corollary}
	\begin{proof}
		The interpretation of $\llbracket u:U. \lbrace u \rbrace \rrbracket$ is some morphism $f$. We have the following equality.
		\begin{equation}
			\llbracket z: \call{P}U. \exists u:U, z = \lbrace u \rbrace \rrbracket = \llbracket z : \call{P}U. \exists u : U, z = fu\rrbracket
		\end{equation}
		Hence the claim follows from Lemma \ref{lem:image_interpretation}.
	\end{proof}
	\subsection{Crucial terms}\label{sec:crucial_terms}
	As mentioned in the Introduction, the union in the internal logic is an \emph{intensional} operator, in that only the union of subobjects of a common parent object can be taken. This is in contrast to ZF set theory where the union of two arbitrary sets may be taken. We now show how to describe the union of two subobjects of the same object via the internal logic. The terms defined here will be used extensively in Section \ref{sec:arbitrary_topos}.
	\begin{definition}\label{def:crucial_terms}
		We define:
		\begin{itemize}
			\item Given two variables $Z_1,Z_2$ both of type $\call{P}U$ we form
			\begin{align}
					Z_1 \cup_U Z_2 &:= \lbrace u:U \mid u \in Z_1 \vee u \in Z_2\rbrace : \call{P}U\\
					\operatorname{FV}(Z_1 \cup_U Z_2) &= \lbrace Z_1 : \call{P}U, Z_2 : \call{P}U\rbrace 
			\end{align}
			\item Given a function symbol $f: A \lto B$ and a variable $Z : \call{P}A$ we define
			\begin{align}
					f(Z) &:= \lbrace b : B \mid \exists a : A, a \in Z \wedge b = fa\rbrace: \call{P}B\\
					 \operatorname{FV}(f(Z)) &= \lbrace Z: \call{P}A\rbrace
			\end{align}
		\end{itemize}
	\end{definition}
	What is the interpretation of these terms inside a topos? We build this piece by piece. Recall that the morphism $\in_U$ is given by first considering the morphism $\adj{\operatorname{id}_{\Omega^U}}: U \times \Omega^U \lto \Omega$ (see Notation \ref{not:adj}) of the identity morphism $\operatorname{id}_U: \Omega^U \lto \Omega^U$. This is a morphism into the subobject classifier and so corresponds to a subobject of $U \times \Omega^U$. We denote this subobject $\in_U$.
	\begin{example}\label{ex:union}
		We let $\Delta$ be the context $u:U, Z_1:\Omega^U, Z_2: \Omega^U$. We construct $\llbracket Z_1:\call{P}U, Z_2:\call{P}U. \lbrace u : U \mid u \in Z_1 \vee u \in Z_2 \rbrace \rrbracket : \Omega^U \times \Omega^U \lto \Omega^U$.
		\begin{enumerate}
			\item The subobject $\llbracket \Delta . u \in Z_1 \rrbracket$ is given by first considering the pullback of the following Diagram, we write $\pi_{(12)}$ for the morphism $U \times \Omega^U \times \Omega^U \lto U \times \Omega^U$, the projection onto the first two components.
			\begin{equation}
				\begin{tikzcd}[column sep = huge]
					& & \in_U\arrow[d, rightarrowtail]\\
					U \times \Omega^U \times \Omega^U\arrow[rr,"{(\llbracket \Delta. u\rrbracket\pi_{(12)}, \llbracket \Delta. Z_1 \rrbracket\pi_{(12)})}"] & & U \times \Omega^U
				\end{tikzcd}
			\end{equation}
			This subobject $\llbracket \Delta. u \in Z_1\rrbracket$ can be thought of as all pairs $(u, U)$ where $u$ ``is an element" of $U$.
			\item The construction of $\llbracket \Delta. u \in Z_2\rrbracket$ is similar, but the morphism $\pi_{(13)}$, the projecto onto the first and third components, is used in place of $\pi_{(12)}$.
			\item The subobject $\llbracket \Delta. u \in Z_1 \vee u \in Z_2\rrbracket$ is given by the coproduct of $\llbracket \Delta. u \in Z_1 \rrbracket$ and $\llbracket \Delta. u \in Z_2 \rrbracket$ in the category $\operatorname{Sub}(U \times \Omega^U \times \Omega^U)$. This subobject can be thought of as all triples $(u,U_1,U_2)$ where $u \in U_1 \cup U_2$.
			\item Lastly, the morphism $\llbracket Z_1:\call{P}U, Z_2:\call{P}U. \lbrace u:U \mid u \in Z_1 \vee u \in Z_2 \rbrace \rrbracket$ is the transpose of the morphism $U \times \Omega^U \times \Omega^U \lto \Omega$ with codomain $\Omega$ which in turn corresponding to the subobject $\llbracket \Delta. u \in Z_1 \vee u \in Z_2 \rrbracket$. Hence
			\begin{equation}
				\llbracket Z_1:\call{P}U, Z_2:\call{P}U. \lbrace u : U \mid u \in Z_1 \vee u \in Z_2 \rbrace \rrbracket : \Omega^U \times \Omega^U \lto \Omega^U
			\end{equation}
			This morphism can be thought of as a map which given two ``set" returns their union.
		\end{enumerate}
	\end{example}
	\begin{example}
		Let $\Delta$ be the context $a:A, b:B, Z:\call{P}A$. We construct $ \llbracket Z:\call{P}A . \lbrace b:B \mid \exists a:A, a \in Z \wedge b = fa \rbrace\rrbracket$.
		\begin{enumerate}
			\item We have already shown in Example \ref{ex:union} how to construct $\llbracket \Delta. a \in Z\rrbracket$.
			\item The subobject $\llbracket \Delta. b = fa\rrbracket$ is given by the following equaliser.
			\begin{equation}
				\begin{tikzcd}
					& & A\arrow[dr,"{f}"]\\
					\operatorname{Equal}(f\pi_A, \pi_B)\arrow[r] & A \times B \times \Omega^A\arrow[ur,"{\pi_A}"]\arrow[rr,swap,"{\pi_B}"] & & B
				\end{tikzcd}
			\end{equation}
			This subobject can be thought of as pairs $(a,f(a))$.
			\item We then take the intersection in the category $\operatorname{Sub}(A \times B \times \Omega^A)$ of these two subobjects to obtain $\llbracket \Delta. a \in Z \wedge b = fa\rrbracket$. This can be thought of as the set of triples $(a,f(a),U)$ where $a \in U$.
			\item The subobject $\llbracket a:A, U: \call{P}A. \exists a:A, a \in Z \wedge b = fa\rrbracket$ is the image of this subobject given the functor $\exists_{\pi_{B \times \Omega^A}}$. This can be thought of as the set of pairs $(f(a),U)$ where $a$ is some element of $U$.
			\item Lastly, we consider the corresponding morphism $B \times \Omega^A \lto \Omega$ with domain $\Omega$ and then take the transpose $\Omega^A \lto \Omega^B$. This can be thought of as a function which maps a subset $U$ of $A$ to the set $f(U)$. Hence why we denote this term $f(Z)$.
		\end{enumerate}
	\end{example}
	
	\section{Finite colimits in an arbitrary topos}\label{sec:arbitrary_topos}
	We lift the ideas of Section \ref{sec:in_sets} to the setting of an arbitrary topos admitting countably infinite colimits.
	
	We begin with the description of an initial object. Then we show how finite coproducts and coequalisers are described. This is then put together in Section \ref{sec:colimits} where we show how to describe finite colimits in an arbitrary topos $\call{E}$ admitting countably infinite colimits using the internal logic.
	
	\subsection{Initial object}\label{sec:initial_object}
	In the topos $\underline{\operatorname{Sets}}$ an initial object is given by the empty set. To translate that this to the internal logic of a topos we consider the formula $\bot$ and then interpret this with respect to the context $\ast: \mathbbm{1}$.
	\begin{equation}
		\begin{tikzcd}
			\mathbbm{1}\\
			\llbracket \ast: \mathbbm{1} . \bot \rrbracket\arrow[u,rightarrowtail]
		\end{tikzcd}
	\end{equation}
	Now, given an arbitrary object $U \in \call{E}$ we claim there exists a unique morphism $\llbracket \ast: \mathbbm{1} . \bot \rrbracket \lto U$. We notice that since $\mathbbm{1}$ is interpreted as a \emph{terminal} object of $\call{E}$, there is no reason to expect there to be a morphism $\mathbbm{1} \lto U$ for arbitrary $U$. However, there is a canonical morphism $\mathbbm{1} \lto \Omega^U$ which is given by $\llbracket . \varnothing_U\rrbracket$. Combining this with the fact that $U$ can be exhibited as a subobject of $\Omega^U$ via the term $\big\lbrace z : \call{P}U \mid z = \lbrace u \rbrace \big\rbrace$ (Corollary \ref{cor:embedded_subobject}) we will be able to arrive at the required morphism $\llbracket \ast: \mathbbm{1} . \bot \rrbracket \lto U$.
	
	Consider the following Diagram.
	\begin{equation}\label{eq:initial_pullback}
		\begin{tikzcd}
			\mathbbm{1}\arrow[r,"{\llbracket . \varnothing_U \rrbracket}"] & \Omega^U\\
			\llbracket \ast: \mathbbm{1} . \bot \rrbracket\arrow[u,rightarrowtail] & \llbracket z: \call{P}U . \exists u:U, z = \lbrace u \rbrace \rrbracket\arrow[u,rightarrowtail]
		\end{tikzcd}
	\end{equation}
	To show that there exists a morphism $\llbracket \ast: \mathbbm{1} . \bot \rrbracket \lto \llbracket z: \call{P}U . \exists u:U, z = \lbrace u \rbrace \rrbracket$ it suffices by Lemma \ref{lem:sub_prop} to derive the following sequent in the internal logic.
	\begin{equation}\label{eq:entailment}
		\bot \vdash \varnothing_U \in \lbrace z: \call{P}U . \exists u:U, z = \lbrace u \rbrace \rbrace
	\end{equation}
	This is immediate though via \ref{rule:true_false} of Definition \ref{def:type_theory}.
	
	In more detail, Lemma \ref{lem:sub_prop} implies that \eqref{eq:initial_pullback} but with the left most vertical arrow replaced by $\llbracket \ast:\mathbbm{1}. \lbrace z: \call{P}U . \exists u:U, z = \lbrace u \rbrace \rbrace\rrbracket$ is a pullback square. Then, since \eqref{eq:entailment} is derivable in the type theory, the subobject $\llbracket \ast:\mathbbm{1}. \bot\rrbracket$ factors through $\llbracket \ast:\mathbbm{1}. \lbrace z: \call{P}U . \exists u:U, z = \lbrace u \rbrace \rbrace\rrbracket$.
	\begin{proposition}\label{prop:initial}
		The domain of any representative of the subobject $\llbracket \ast: \mathbbm{1} . \bot \rrbracket$ is an initial object of $\call{E}$.
	\end{proposition}
	\begin{proof}
		We have already shown that for arbitrary $U \in \call{E}$ there exists a morphism $\llbracket \ast : \mathbbm{1} . \bot \rrbracket \lto U$, so it remains to show uniqueness of this morphism.
		
		To prove this, it suffices to show that $\llbracket . \varnothing_U\rrbracket: \mathbbm{1} \lto \Omega^U$ is a monomorphism. However this is immediate from the fact that the domain of this morphism is $\mathbbm{1}$, a terminal object of $\call{E}$.
	\end{proof}
	
	\subsection{Finite coproducts}\label{sec:coproducts}
	Throughout, the notation $\lbrace a \rbrace = \lbrace a' : A. a' = a\rbrace$, and $\varnothing_{a:A} = \lbrace a:A . \bot\rbrace$ will be used.
	
	Consider the following formula which has a free variable $z: \call{P}A \times \call{P}B$.
	\begin{equation}\label{eq:coprod_formula}
		\big(A \coprod B\big)(z) := \big(\exists a:A, z = \langle \lbrace a \rbrace, \varnothing_{b:B}\rangle \big) \vee \big(\exists b : B, z = \langle \varnothing_{a:A}, \lbrace b \rbrace\rangle\big)
	\end{equation}
	The interpretation of this is a subobject which fits into the following Diagram of solid arrows.
	\begin{equation}\label{eq:coproduct_diagram}\adjustbox{scale=0.75}{
		\begin{tikzcd}[column sep = huge]
			A\arrow[r,"{\llbracket a:A. \lbrace a \rbrace \rrbracket}"]\arrow[drr,swap,dashed,"{\iota_A}"] & \Omega^A\arrow[r,"{\llbracket z: \call{P}A . \langle z, \varnothing_{b:B} \rangle \rrbracket}"] & \Omega^A \times \Omega^B & \Omega^B\arrow[l,swap,"{\llbracket z: \call{P}B . \langle \varnothing_{a:A}, z\rangle \rrbracket}"] & B\arrow[l,swap,"{\llbracket b:B. \lbrace b \rbrace\rrbracket}"]\arrow[dll,dashed,"{\iota_B}"]\\
			& & \llbracket z: \call{P}A \times \call{P}B. \big(A \coprod B\big)(z)\rrbracket \arrow[u,rightarrowtail]
		\end{tikzcd}}
	\end{equation}
	We now wish to invoke Lemma \ref{lem:sub_prop} to infer the existence of morphisms $\iota_A: A \lto \llbracket z:\call{P}A \times \call{P}B. \big(A \coprod B\big)(z)\rrbracket$ and $\iota_B: B \lto \llbracket z:\call{P}A \times \call{P}B. \big(A \coprod B\big)(z)\rrbracket$ rendering Diagram \eqref{eq:coproduct_diagram} commutative. Using this Lemma, it suffices to prove the following sequents in the type theory.
	\begin{align}
		&\vdash_{a:A} \langle \lbrace a \rbrace, \varnothing_{b:B}\rangle \in \lbrace z: \call{P}A \times \call{P}B\mid \big(A \coprod B\big)(z)\rbrace,\label{eq:inclusion_A}\\
		&\vdash_{b:B} \langle \varnothing_{a:A}, \lbrace b \rbrace\rangle \in \lbrace z: \call{P}A \times \call{P}B\mid \big(A \coprod B\big)(z)\rbrace
	\end{align}
	We prove that Sequent \eqref{eq:inclusion_A} is derivable.
	
	Using deduction rule \eqref{rule:element} it suffices to prove the following.
	\begin{equation}
		\vdash_{a:A} \big(A \coprod B\big)(\langle \lbrace a \rbrace, \varnothing_{b:B}\rangle)
	\end{equation}
	To prove this, we observe the following proof tree.
	\begin{scprooftree}{0.8}
		\AxiomC{$\langle \lbrace a \rbrace , \varnothing_{b:B}\rangle = \langle \lbrace a \rbrace , \varnothing_{b:B}\rangle \vdash \langle \lbrace a \rbrace , \varnothing_{b:B}\rangle = \langle \lbrace a \rbrace , \varnothing_{b:B}\rangle$}
		\RightLabel{Lemma \ref{lem:disjunction_right}}
		\UnaryInfC{$\langle \lbrace a \rbrace , \varnothing_{b:B}\rangle = \langle \lbrace a \rbrace , \varnothing_{b:B}\rangle \vdash_{a:A,b:B} \langle \lbrace a \rbrace , \varnothing_{b:B}\rangle = \langle \lbrace a \rbrace , \varnothing_{b:B}\rangle \vee \langle \lbrace a \rbrace, \varnothing_{b:B}\rangle = \langle \varnothing_{a:A}, \lbrace b \rbrace\rangle$}
		\RightLabel{Lemma \ref{lem:witness}}
		\UnaryInfC{$\langle \lbrace a \rbrace , \varnothing_{b:B}\rangle = \langle \lbrace a \rbrace , \varnothing_{b:B}\rangle \vdash_{a:A,b:B} \exists z: \call{P}A \times \call{P}B, z = \langle \lbrace a \rbrace, \varnothing_{b:B}\rangle \vee z = \langle \varnothing_{a:A}, \lbrace b \rbrace \rangle$}
	\end{scprooftree}
	Let $\pi$ denote this proof tree just constructed.  We can then construct the following proof tree.
	\begin{center}
		\AxiomC{}
		\RightLabel{$\eqref{rule:equality}$}
		\UnaryInfC{$\vdash_{a:A,b:B}\langle \lbrace a \rbrace, \varnothing_{b:B}\rangle = \langle \lbrace a \rbrace, \varnothing_{b:B}\rangle$}
		\AxiomC{$\pi$}
		\noLine
		\UnaryInfC{$\vdots$}
		\RightLabel{$\eqref{rule:cut}$}
		\BinaryInfC{$\vdash_{a:A,b:B} \exists z: \call{P}A \times \call{P}B, z = \langle \lbrace a \rbrace, \varnothing_{b:B}\rangle \vee z = \langle \varnothing_{a:A}, \lbrace b \rbrace \rangle$}
		\DisplayProof
	\end{center}
	Hence the morphism $\iota_A$ exists. The existence of morphism $\iota_B$ is proved similarly.
	
	\begin{proposition}\label{prop:coproduct}
		The triple $(\llbracket z: \call{P}A \times \call{P}B . \big(A \coprod B\big)(z) \rrbracket, \iota_A, \iota_B)$ forms a coproduct of $A,B$.
	\end{proposition}
	\begin{proof}
		The proof will proceed with two arguments, the first establishes the existence of a morphism
		\begin{equation}
			\llbracket z : \call{P}A \times \call{P}B . (A \coprod B)(z)\rrbracket \lto U
		\end{equation}
		rendering the following Diagram commutative.
		\begin{equation}\label{eq:standard_coprod_diag}
			\begin{tikzcd}
				& U\\
				A\arrow[ur,"{g_0}"]\arrow[r,swap,"{\iota_A}"] & \llbracket Z : \call{P}A \times \call{P}B . (A \coprod B)(z)\rrbracket\arrow[u] & B\arrow[ul,swap,"g_1"]\arrow[l,"{\iota_B}"]
			\end{tikzcd}
		\end{equation}
		and the second argument establishes uniqueness of such a morphism.
		
		To show existence, we consider the following subobject
		\begin{equation}
			\begin{tikzcd}
				\Omega^U\\
				\llbracket z: \call{P}U. \exists u:U,  z = \lbrace u \rbrace \rrbracket\arrow[u,rightarrowtail]
			\end{tikzcd}
		\end{equation}
		which, by Corollary \ref{cor:embedded_subobject} is isomorphic to the object $U$. We then define the following terms, where $z: \call{P}A \times \call{P}B$ and $u: \call{P}U \times \call{P}U$ (see Definition \ref{def:crucial_terms}).
		\begin{equation}
			\langle g_0(\operatorname{fst}z),g_1(\operatorname{snd}z)\rangle\qquad  \operatorname{fst}u \cup_U \operatorname{snd}u 
		\end{equation}
		These respectively have interpretations which are morphisms $\Omega^A \times \Omega^B \lto \Omega^U \times \Omega^U$ and $\Omega^U \times \Omega^U \lto \Omega^U$. This data fits into the following Diagram of solid arrows.
		\begin{equation}\label{eq:existence_simple_diagram}\adjustbox{scale=0.75}{
			\begin{tikzcd}[column sep = huge]
				\Omega^A \times \Omega^B\arrow[rr, "{\llbracket z: \call{P}A \times \call{P}B.  \langle g_0(\operatorname{fst}z),g_1(\operatorname{snd}z)\rangle\rrbracket}"] & & \Omega^U \times \Omega^U\arrow[r,"{\llbracket u: \call{P}U \times \call{P}U . \operatorname{fst}u \cup_U \operatorname{snd}u \rrbracket}"] & \Omega^U\\
				\llbracket z: \call{P}A \times \call{P}B. \big(A \coprod B\big)(z) \rrbracket\arrow[u,rightarrowtail]\arrow[rrr,dashed] & & & \llbracket z:\call{P}U . \exists u:U, z = \lbrace u \rbrace \rrbracket\arrow[u,rightarrowtail]
			\end{tikzcd}}
		\end{equation}
		On the level of the type theory the following sequent is derivable.
		\begin{equation}
			\big(A \coprod B\big)(z) \vdash_{z:\call{P}A \times \call{P}B} g_0(\operatorname{fst}z) \cup_U g_1(\operatorname{snd}z) \in \lbrace v:\call{P}U \mid v = \lbrace u \rbrace \rbrace
		\end{equation}
		which by Lemma \ref{lem:sub_prop} induces a morphism
		\begin{equation}
			\llbracket z: \call{P}A \times \call{P}B. \big(A \coprod B\big)(z) \rrbracket \lto \llbracket z:\call{P}U . \exists u: U,z = \lbrace u \rbrace \rrbracket
		\end{equation}
		which renders Diagram \eqref{eq:existence_simple_diagram} (this time considering all arrows) commutative.
		
		Now we show uniqueness. Since $U \lto \Omega^U$ is monic, it suffices to show uniqueness of a morphism $\psi: \llbracket z: \call{P}A \times \call{P}B.(A \coprod B)(z) \rrbracket \lto \Omega^U$ rendering the following Diagram commutative.
		\begin{equation}\label{eq:sufficient_commutativity}
			\begin{tikzcd}
				& \Omega^U\\
				A\arrow[ur,"{\llbracket a:A. \lbrace g_0a \rbrace\rrbracket}"]\arrow[r,swap,"{\iota_A}"] & \llbracket z : \call{P}A \times \call{P}B . (A \coprod B)(z)\rrbracket\arrow[u,"{\psi}"] & B\arrow[ul,swap,"{\llbracket b:B. \lbrace g_1b \rbrace\rrbracket}"]\arrow[l,"{\iota_B}"]
			\end{tikzcd}
		\end{equation}
		Let $\psi$ be such a morphism. By commutativity of Diagram \eqref{eq:sufficient_commutativity}, the following are derived sequents (Definition \ref{def:sound_deduction_rules}).
		\begin{align*}
			\exists a:A, z = \langle \lbrace a \rbrace, \varnothing \rangle &\vdash_{z:\call{P}A \times \call{P}B}\psi(z) = \lbrace g_0 a \rbrace\\
			\exists b:B, z = \langle \varnothing, \lbrace b \rbrace \rangle &\vdash_{z: \call{P}A \times \call{P}B} \psi(z) = \lbrace g_1 b \rbrace
		\end{align*}
		We then derive the following sequents:
		\begin{align*}
			\exists a:A, z = \langle \lbrace a \rbrace, \varnothing \rangle &\vdash_{z: \call{P}A \times \call{P}B}\lbrace g_0a\rbrace = g_0(\operatorname{fst}z)\cup g_1(\operatorname{snd}z)\\
			\exists a:A, z = \langle \varnothing, \lbrace b \rbrace \rangle &
			\vdash_{z: \call{P}A \times \call{P}B}\lbrace g_1b\rbrace = g_0(\operatorname{fst}z)\cup g_1(\operatorname{snd}z)
		\end{align*}
		Hence we derive by \eqref{rule:disjunction} of \ref{def:type_theory}, the following sequent.
		\begin{equation}\label{eq:psi}
			\exists a:A, z = \langle \lbrace a \rbrace, \varnothing \rangle \vee \exists b:B, z = \langle \varnothing, \lbrace b \rbrace \rangle \vdash_{z: \call{P}A \times \call{P}B}\psi(z) = g_0(\operatorname{fst}z) \cup g_1(\operatorname{snd}z)
		\end{equation}
		Since Sequent \eqref{eq:psi} is derivable, we have that the following diagram is commutative.
		\begin{equation}\adjustbox{scale=0.8}{
			\begin{tikzcd}[column sep = huge]
				\Omega^A \times \Omega^B\arrow[rr, "{\llbracket z: \call{P}A \times \call{P}B.  \langle g_0(\operatorname{fst}z),g_1(\operatorname{snd}z)\rangle\rrbracket}"] & & \Omega^U \times \Omega^U\arrow[rr,"{\llbracket u: \call{P}U \times \call{P}U . \operatorname{fst}u \cup_U \operatorname{snd}u \rrbracket}"] & & \Omega^U\\
				\llbracket z: \call{P}A \times \call{P}B. \big(A \coprod B\big)(z) \rrbracket\arrow[u,rightarrowtail]\arrow[rrrru,swap,"{\psi}"]
			\end{tikzcd}}
		\end{equation}
		Since the top row and the vertical arrow are independent of $\psi$, this proves uniqueness.
	\end{proof}
	\begin{remark}
		The top row of Diagram \eqref{eq:existence_simple_diagram} is exactly the top row of the Diagram in the middle of page 27 of \cite{Mikkelson}. We note however that we were not aware of this paper when the Diagram \eqref{eq:existence_simple_diagram} was first constructed.
	\end{remark}
The above is \emph{not} sufficient in order to prove that an elementary topos admits finite coproducts, the reason why is because the formula \eqref{eq:coprod_formula} makes use of the conjunction $\wedge$ which according to Definition \ref{def:interpretation} is interpreted as the disjoint union in a subobject category. To ensure that this exists, we use the fact a priori that an elementary topos admits finite colimits.

Indeed, the point is to come up with a \emph{description} of these finite coproducts using the internal language. Hence, we cannot simply use the fact that we have a description of \emph{binary} coproducts in order to obtain for free a descritpion of \emph{arbitrary finite} coproducts.

The next definition gives an appropriate description, the fact that this is correct follows by making suitable changes to the case of binary coproducts.
	\begin{definition}
		Let $A_1,...,A_n$ be objects of a topos $\call{E}$ and consider them as types in the pure type theory \eqref{def:type_theory}. Consider the following formula which has a free variable $z: (\hdots (\call{P} A_1 \times \call{P} A_2) \times \hdots \times \call{P}A_n)$.
		\begin{align}
			\begin{split}
				\big( \hdots \big(A_1 \coprod A_2\big) &\coprod \hdots  \coprod A_n \big)(z) \\
			&:= \big( \exists a_1: A_1, z = \langle \hdots \langle \lbrace a \rbrace, \varnothing_{a_2: A_2}\rangle, \hdots \rangle , \varnothing_{a_n: A_n}\rangle\big)\\
			&\vee \hdots\\
			&\vee \big(\exists a_n: A_n, z = \langle \hdots \langle \varnothing_{a_1: A_1}, \varnothing_{a_2: A_2}\rangle, \hdots \rangle, \lbrace a_n\rbrace\rangle\big)
			\end{split}
			\end{align}
		Also, for each $i = 1,..., n$, we let $\iota_{a_i}$ denote the following term.
		\begin{equation}
			\iota_{a_i} := \langle \hdots \langle \varnothing_{a_1:A_1}, \hdots \rangle,  \lbrace a_i\rbrace \rangle, \hdots \rangle, \varnothing_{a_n: A_n}\rangle
			\end{equation}
	\end{definition}
\begin{proposition}
	The tuple $\big(\llbracket z: \prod_{i = 1}^n \call{P} A_i. (A_1 \coprod A_2) \coprod \hdots \coprod A_n \rrbracket, \lbrace \iota_{a_i}\rbrace_{i = 1}^n \big)$ is a coproduct of $A_1,...,A_n$.
	\end{proposition}

	\subsection{Coequalisers}
	\label{sec:coequaliser}
	In the topos \underline{Sets} the coequaliser of functions $f,g: A \to B$ is given by $B/R$, where $R$ is the smallest equivalence relation on $B$ such that $(f(a), g(a)) \in R$, for all $a \in A$. We emulate this by taking the subobject of $\Omega^B$ consisting of ``equivalence classes" of $R$. First though, we define a term which simulates $R$:
	\begin{definition}
		\label{def:relation}
		Given terms $t_0,t_1: A \to B$ with a free variable $a:A$ in a type theory, let $R_{t_0,t_1}$ be the term
		\begin{align*}
			\Big\lbrace &z: B \times B \mid \exists b_1, b_2 : B\text{, }z = \langle b_1, b_2 \rangle \wedge \Big(b_1 = b_2\\
			&\vee \bigvee_{n = 1}^\infty\bigvee_{\alpha \in \bb{Z}_2^n}\big( \exists a_1,...,a_n, (b_1 = t_{\alpha_0}[a:= a_1]\\
			&\wedge (t_{\alpha_1 + 1}[a:= a_1] = t_{\alpha_1}[a:= a_2]) \wedge ... \wedge (t_{\alpha_{n-1}+1}[a:= a_{n-1}] = t_{\alpha_{n}}[a:= a_n])\\
			&\wedge (t_{\alpha_{n}+1}[a:= a_n] = b_2)\big)\Big)\Big\rbrace
		\end{align*}
		where $\mathbb{Z}_2^n$ is the set of length $n$ sequences of elements of $\mathbb{Z}_{2}$. 
	\end{definition}
	\begin{remark}
		As mentioned in the Introduction, we require that every object $C$ in the topos in question $\call{E}$ satisfies the property that the category $\operatorname{Sub}(C)$ admits countable disjoint union (this is satisfied, for example, when the topos $\call{E}$ admits all countably infinite colimits). This is needed so that the term given in Definition \ref{def:relation} admits an interpretation inside $\call{E}$ (see Remark \ref{rmk:countable_colimits})
	\end{remark}
	
	The idea of this term is that it is a generalisation of the smallest equivalence relation $\sim$ such that for all $a$, $t_0(a) \sim t_1(a)$. It does this by declaring $\langle b_1,b_2\rangle$ to be in the relation if either $b_1 = b_2$, or there exists a sequence $(a_1,...,a_{n})$ which ``connect" $b_1$ and $b_2$ by the images of these $a_i$ under $t_0$ and $t_1$. A diagram representing an example of this is the following,
	
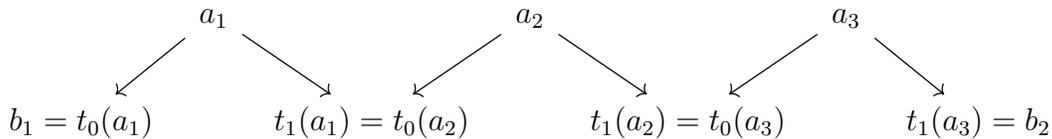
\[
	\begin{tikzcd}[column sep = tiny]
		& a_1\arrow[dl]\arrow[dr] & & a_2\arrow[dl]\arrow[dr] & & a_3\arrow[dl]\arrow[dr]\\
		b_1 = t_0(a_1) & & t_1(a_1) = t_0(a_2) & & t_1(a_2) = t_0(a_3) & & t_1(a_3) = b_2
	\end{tikzcd}
	\]
	However, it must also be allowed for that $t_0$ and $t_1$ do not always appear in this order, due to the symmetry axiom of an equivalence relation, this is why the set $\bb{Z}_2^n$ is considered. The next Definition generalises ``the set of elements related to $b$ under $R_{t_0,t_1}$":
	\begin{definition}
		In the setting of Definition \ref{def:relation}, for any variable $b:B$, define the following term.
		\[[b]_{t_0,t_1} := \lbrace b' : B \mid \langle b, b'\rangle \in R_{t_0,t_1} \rbrace\]
	\end{definition}
	Now to define the coequaliser.
	\begin{definition}
		Let $t_0,t_1$ be terms of type $B$ with $\operatorname{FV}(t_0) = \operatorname{FV}(t_1) = \lbrace a: A \rbrace$. We define the following formula with free variable $z : \call{P}B$.
		\begin{equation}\label{eq:coeq_def}
			\operatorname{Coeq}(t_0,t_1)(z) := \exists b: B, z = [b]_{t_0,t_1}
		\end{equation}
	\end{definition}
	By Lemma \ref{lem:sub_prop}, if the following sequent is provable:
	\begin{equation}\label{eq:coeq_sequent}
		\vdash_{b:B}[b] \in \lbrace z: \call{P}B . \operatorname{Coeq}(t_0,t_1)(z)\rbrace
	\end{equation}
	then there exists a morphism $B \lto \llbracket z:\call{P}B . \operatorname{Coeq}(t_0,t_1)(z)\rrbracket$ such that the following Diagram commutes.
	\begin{equation}\label{eq:coeq_existence_commute}
		\begin{tikzcd}
			B\arrow[r,"{\llbracket b:B. [b]\rrbracket}"]\arrow[dr,swap,"{c}"] & \Omega^B\\
			& \llbracket z:\call{P}B. \operatorname{Coeq}(t_0,t_1)(z)\rrbracket\arrow[u,rightarrowtail]
		\end{tikzcd}
	\end{equation}
	Sequent \eqref{eq:coeq_sequent} is clearly provable, because by the deduction rule \eqref{rule:element} it suffices to derive the following.
	\begin{equation}
		\vdash_{b:B}\operatorname{Coeq}(t_0,t_1)([b])
	\end{equation}
	This follows from the following statement.
	\begin{equation}
		\vdash_{b:B}[b] = [b]
	\end{equation}
	Hence there exists the morphism $c$ in \eqref{eq:coeq_existence_commute}.
	\begin{proposition}
		The pair $(\llbracket z: \call{P}B. \operatorname{Coeq}(t_0,t_1)(z)\rrbracket,c)$ is a coequaliser of $\llbracket a:A. t_0\rrbracket, \llbracket a:A. t_1\rrbracket$.
	\end{proposition}
	\begin{proof}
		We let $f_0,f_1$ respectively denote $\llbracket a:A. t_0\rrbracket, \llbracket a:A, t_1\rrbracket$. Say $\varphi: B \lto C$ is a morphism such that $\varphi f_0 = \varphi f_1$. First we prove the existence of a morphism $\llbracket z: \call{P}B. \operatorname{Coeq}(t_0,t_1)(z)\rrbracket \lto C$ rendering the following Diagram commutative.
		\begin{equation}
			\begin{tikzcd}[column sep = large]
				B\arrow[dr,swap,"c"]\arrow[r,"{\llbracket b:B. [b]\rrbracket}"] & \llbracket z: \call{P}B. \operatorname{Coeq}(t_0,t_1)(z)\rrbracket\arrow[d]\\
				& C
			\end{tikzcd}
		\end{equation}
		By Corollary \ref{cor:embedded_subobject} there is an isomorphism $C \lto \llbracket z: \Omega^C \mid \exists c:C, z = \lbrace c \rbrace \rrbracket$ rendering the following Diagram commutative.
		\begin{equation}
			\begin{tikzcd}[column sep = large]
				C\arrow[r,"{\llbracket c:C. \lbrace c \rbrace \rrbracket}"]\arrow[dr] & \Omega^C\\
				& \llbracket z: \call{P}C . \exists c:C, z = \lbrace c \rbrace \rrbracket\arrow[u,rightarrowtail]
			\end{tikzcd}
		\end{equation}
		Hence it suffices to construct a morphism
		\begin{equation}
			\llbracket z: \call{P}B . \operatorname{Coeq}(t_0,t_1)(z)\rrbracket \lto \llbracket z: \call{P}C. \exists c: C, z = \lbrace c \rbrace \rrbracket
		\end{equation}
		We construct this by considering the following commutative Diagram of solid arrows. Recall from Definition \ref{def:crucial_terms} that $\varphi(z)$, where $z:\call{P}B$ is a variable, is notation for the term $\lbrace c:C \mid \exists b:B, b \in z \wedge \varphi(b) = c\rbrace$.
		\begin{equation}\label{eq:coeq_embedding}
			\begin{tikzcd}
				\Omega^B\arrow[r,"{\llbracket z:\call{P}B. \varphi(z)\rrbracket}"] & \Omega^C\\
				\llbracket z: \call{P}B . \operatorname{Coeq}(t_0,t_1)(z)\rrbracket\arrow[u,rightarrowtail]\arrow[r,dashed] & \llbracket z: \call{P}C. \exists c: C, z = \lbrace c \rbrace \rrbracket\arrow[u,rightarrowtail]
			\end{tikzcd}
		\end{equation}
		Then we appeal to Lemma \ref{lem:sub_prop} and reduce to deriving the following sequent.
		\begin{equation}
			\operatorname{Coeq}(t_0,t_1)(z)\vdash_{z:\call{P}B}\varphi(z) \in \lbrace w: \call{P}C, \exists c:C, w = \lbrace c \rbrace\rrbracket
		\end{equation}
		Which by the derived deduction rule \eqref{rule:element} reduces to deriving the following, at this point we also make use of \eqref{eq:coeq_def}.
		\begin{equation}
			\exists b:B, z = [b] \vdash_{z: \call{P}B}\exists c:C, \varphi(z) = \lbrace c \rbrace
		\end{equation}
		By \eqref{rule:existential} of Definition \ref{def:type_theory} and Lemma \ref{lem:variablesub} it suffices to derive the following.
		\begin{equation}
			\vdash_{b:B}\exists c:C, \varphi([b]) = \lbrace c \rbrace
		\end{equation}
		By Lemma \ref{lem:witness} it suffices to derive the following.
		\begin{equation}\label{eq:seemingly_dodgy}
			\vdash_{b:B}\varphi([b]) = \lbrace \varphi(b)\rbrace
		\end{equation}
		In the left hand side of the equality in \eqref{eq:seemingly_dodgy} the term $\varphi([b])$ means the term $\lbrace c:C \mid \exists b':B, c \in [b] \wedge \varphi(c) = b'\rbrace$ and the term on the right hand side of the equality means the term $\lbrace c :C \mid c = \varphi(b) \rbrace$. Hence we must derive the following.
		\begin{equation}
			\vdash_{b:B} \lbrace c:C \mid \exists b':B, b' \in [b] \wedge \varphi(b') = c\rbrace = \lbrace c :C \mid c = \varphi(b) \rbrace
		\end{equation}
		To derive this, it is sufficient by the derived deduction rule \eqref{rule:set} and the deduction rules \eqref{rule:implication}, \eqref{rule:conjunction} to derive the following sequents.
		\begin{align}
			\exists b':B, b' \in [b] \wedge \varphi(b') = c&\vdash_{b : B, c:C} c = \varphi(b)\\
			c = \varphi(b) &\vdash \exists b':B, b' \in [b] \wedge \varphi(b') = c
		\end{align}
		First we derive $\exists b':B, c \in [b] \wedge \varphi(b') = c\vdash_{b : B, c:C} c = \varphi(b)$. The term $[b]$ means $\lbrace b' :B \mid \langle b,b'\rangle \in R_{t_0,t_1}\rbrace$. Using this, along with \eqref{rule:existential} of Definition \ref{def:type_theory}, \eqref{rule:element}, and Lemma \ref{lem:variablesub} it remains to derive the following.
		\begin{equation}
			\langle b,b'\rangle \in R_{t_0,t_1} \vdash_{b:B,b':B} \varphi(b') = \varphi(b)
		\end{equation}
		Unravelling the Definition of $R_{t_0,t_1}$, we see that it remains to derive the following. We write $t_0a_1$ for $t_0[a := a_1]$, and similarly for $t_1, a_2,\ldots$ etc.
		\begin{align*}
			\exists b_1, b_2 : B\text{, }\langle b,b'\rangle = &\langle b_1, b_2 \rangle \wedge \Big(b_1 = b_2\\
			&\vee \bigvee_{n = 1}^\infty\bigvee_{\alpha \in \bb{Z}_2^n}\big( \exists a_1,...,a_n, (b_1 = t_{\alpha_0}a_1)\\
			&\wedge (t_{\alpha_1 + 1}a_1 = t_{\alpha_1}a_2) \wedge ... \wedge (t_{\alpha_{n-1}+1}a_{n-1} = t_{\alpha_{n}}a_{n})\\
			&\wedge (t_{\alpha_{n}+1}a_{n} = b_2)\big)\Big) \vdash_{b:B,b':B} \varphi(b') = \varphi(b)
		\end{align*}
		Which by \eqref{rule:existential}, Lemma \ref{lem:variablesub} ammounts to deriving the following.
		\begin{align*}
			&b = b' \vee \bigvee_{n = 1}^\infty\bigvee_{\alpha \in \bb{Z}_2^n}\big( \exists a_1,...,a_n, (b = t_{\alpha_0}a_1)\\
			&\wedge (t_{\alpha_1 + 1}a_1 = t_{\alpha_1}a_2) \wedge ... \wedge (t_{\alpha_{n-1}+1}a_{n-1} = t_{\alpha_{n}}a_{n})\\
			&\wedge (t_{\alpha_{n}+1}a_{n} = b')\big) \vdash_{b:B,b':B} \varphi(b') = \varphi(b)
		\end{align*}
		Using the derived rule \eqref{rule:infinite} of Proposition \ref{prop:derived_rules} we now have to show for all $n >0$ and $\alpha \in \bb{Z}_2^n$ that the following sequent is derivable.
		\begin{align}\label{eq:final_leg}
			\begin{split}
				&(b = t_{\alpha_0}a_1) \wedge (t_{\alpha_1 + 1}a_1 = t_{\alpha_1}a_2) \wedge ... \wedge (t_{\alpha_{n-1}+1}a_{n-1} = t_{\alpha_{n}}a_{n}) \wedge (t_{\alpha_{n}+1}a_{n} = b') \\
				&\vdash_{b:B,b':B,a_1:A,...,a_n:A} \varphi(b') = \varphi(b)
				\end{split}
		\end{align}
		We can then show the following.
		\begin{align*}
			&(b = t_{\alpha_0}a_1) \wedge (t_{\alpha_1 + 1}a_1 = t_{\alpha_1}a_2) \wedge ... \wedge (t_{\alpha_{n-1}+1}a_{n-1} = t_{\alpha_{n}}a_{n}) \wedge (t_{\alpha_{n}+1}a_{n} = b') \\
			&\vdash_{b:B,b':B,a_1:A,...,a_n:A}
			(\varphi(b) = \varphi(t_{\alpha_0}a_1)) \wedge (\varphi(t_{\alpha_1 + 1}a_1) = \varphi (t_{\alpha_1}a_2))\\
			& \wedge ... \wedge (\varphi(t_{\alpha_{n-1}+1}a_{n-1}) = \varphi(t_{\alpha_{n}}a_{n})) \wedge (\varphi t_{\alpha_n + 1}a_n = \varphi(b'))
		\end{align*}
		Now, the following sequent is valid in the type theory.
		\begin{equation}
			\vdash_{a:A}\varphi (t_0 a) = \varphi (t_1 a)
		\end{equation}
		Hence, the following sequent is derivable.
		\begin{align*}
			&(b = t_{\alpha_0}a_1) \wedge (t_{\alpha_1 + 1}a_1 = t_{\alpha_1}a_2) \wedge ... \wedge (t_{\alpha_{n-1}+1}a_{n-1} = t_{\alpha_{n}}a_{n}) \wedge (t_{\alpha_{n}+1}a_{n} = b') \\
			&\vdash_{b:B,b':B,a_1:A,...,a_n:A}
			\varphi(b) = \varphi(t_{\alpha_0}a_1) = \varphi(t_{\alpha_1 + 1}a_1) = \varphi(t_{\alpha_1}a_2)\\
			& = ... = \varphi(t_{\alpha_{n-1}+1}a_{n-1}) = \varphi(t_{\alpha_{n}}a_{n}) = \varphi (t_{\alpha_n + 1}a_n) = \varphi(b')
		\end{align*}
		Hence, using the deduction rule \eqref{rule:cut} we finally arrive at Sequent \eqref{eq:final_leg}, as required.
		
		To prove uniqueness, we first notice that it suffices to show that there exists a unique morphism
		\begin{equation}
			\psi: \llbracket z: \call{P}B. \operatorname{Coeq}(t_0,t_1)(z)\rrbracket \lto \llbracket z: \call{P}C. \exists c:C, z = \lbrace c \rbrace \rrbracket
		\end{equation}
		such that Diagram \eqref{eq:coeq_embedding} commutes with the horizontal arrow on the bottom row replaced by $\psi$, as the subobject $\llbracket z: \call{P}C. \exists c:C, z = \lbrace c \rbrace\rrbracket \rightarrowtail \Omega^C$. However this is simple, it follows from the fact that the subobject $\llbracket z: \call{P}B. \operatorname{Coeq}(t_0,t_1)(z)\rrbracket$ as well as the horizontal arrow on the bottom row of Diagram \eqref{eq:coeq_embedding} are independent of $\psi$.
	\end{proof}
	
	\subsection{Finite colimits}\label{sec:colimits}
	Let $J = \lbrace A_1,\hdots A_n\rbrace$ be a finite set of objects in a topos $\call{E}$ and let $\lbrace f_1,\hdots, f_m\rbrace$ be a finite set of morphisms where for each $i=1,\hdots, m$ the morphism $f_i$ has domain $\operatorname{dom}f_i$ and codomain $\operatorname{cod}f_i$.
	First, we construct the two coproducts, which are the following subobjects, where $z_1 : (\hdots(\call{P} \operatorname{dom}f_1 \times \call{P} \operatorname{dom}f_2) \times \hdots \times \call{P} \operatorname{dom}f_n)$ and $( \hdots z_2 : (\call{P} A_1 \times \call{P} A_2 ) \times \hdots \times \call{P}A_n)$.
	\begin{equation}
		\begin{tikzcd}
			\big(\hdots (\Omega^{\operatorname{dom}f_1} \times \Omega^{\operatorname{dom}f_2}) \times \hdots \times \Omega^{\operatorname{dom}f_n}\big)\\
			\big( \hdots ( \operatorname{dom}f_1 \coprod \operatorname{dom}f_2) \coprod \hdots \coprod \operatorname{dom}f_n\big)(z_1)\arrow[u,rightarrowtail]
			\end{tikzcd}
		\end{equation}
	\begin{equation}
		\begin{tikzcd}
			\big(\hdots (\Omega^{A_1} \times \Omega^{A_2}) \times \hdots \times \Omega^{A_n}\big)\\
			\big( \hdots ( A_1 \coprod A_2) \coprod \hdots \coprod A_n\big)(z_2)\arrow[u,rightarrowtail]
		\end{tikzcd}
	\end{equation}
	To ease notation, we let $\prod_{i = 1}^n \Omega^{\operatorname{dom}f_i}$ denote $\big(\hdots (\Omega^{\operatorname{dom}f_1} \times \Omega^{\operatorname{dom}f_2}) \times \hdots \times \Omega^{\operatorname{dom}f_n}\big)$, and similarly for $\prod_{i = 1}^n\Omega^{A_i}, (\coprod_{i = 1}^n \operatorname{dom}f_i)(z_1), (\coprod_{i = 1}^n A_i)(z_2)$.
	
	Consider the sequence of objects $(\operatorname{dom}f_1,...,\operatorname{dom}f_m)$. For each object $A_i \in J$ there exists a unique, maximal subsequence $(k^i_1,...,k^i_{l_i})$ of $(\operatorname{dom}f_1,...,\operatorname{dom}f_m)$ satisfying $k^i_1 = \hdots = k^i_{l_i} = A_i$. We use this subsequence to define a term. Let $W_{k_1^i},...,W_{k_{l_i}^i}$ all be variables of type $\Omega^{A_i}$ and consider the following term.
	\begin{equation}
		W_{A_i} := W_{k_1^i} \cup \hdots \cup W_{k_{l_i}^i}
	\end{equation}
	Taking the product over all $i = 1,...,n$ we obtain the following term.
	\begin{equation}
		t_0 := \langle W_{A_1}, \hdots, W_{A_n}\rangle
	\end{equation}
	We let $\Delta$ denote the context $W_{k_1^1}:\call{P}A_1,...,W_{k_{l_1}^1}:\call{P}A_1,...,W_{k_1^n}:\call{P}A_n,...,W_{k_{l_n}^n}:\call{P}A_n$. That is, the underlying set of $\Delta$ is $\operatorname{FV}(W_{A_1}) \cup \hdots \cup \operatorname{FV}(W_{A_n})$.
	\begin{equation}\label{eq:compression}
		\llbracket \Delta. \langle W_{A_1},...,W_{A_n}\rangle \rrbracket: \prod_{i = 1}^m \Omega^{\operatorname{dom}f_i} \lto \prod_{j = 1}^n \Omega^{A_j}
	\end{equation}
	Similarly, there exists a unique, maximal subsequence $(k_1^i,...,k_{l_i}^i)$ of $(\operatorname{cod}f_1,...,\operatorname{cod}f_m)$ satisfying $k_i^1 = \hdots = k_{l_i}^i = A_i$. We define the following term, where $Y_{k_1^i},...,Y_{k_{l_i}^i}$ are all variables of type $\Omega^{A_i}$.
	\begin{equation}
		Y_{A_i} := Y_{k_1^i} \cup \hdots \cup Y_{k_{l_i}^i}: \call{P}A_i
	\end{equation}
	Note, if there are no morphisms $f_i$ such that $\operatorname{cod}f_i = A_i$ then we take $Y_{A_i} := \varnothing$.
	
	We take the product over $i = 1,...,n$ to obtain the following term.
	\begin{equation}
		\langle Y_{A_1},...,Y_{A_n}\rangle
	\end{equation}
	This induces the following morphism, we denote by $\Gamma$ the context $Y_{k_1^1}:A_1,...,Y_{k_{l_1}^1}:A_1,...,Y_{k_1^n}:A_n,...,Y_{k_{l_n}^n}:A_n$. That is, the underlying set of $\Gamma$ is $\operatorname{FV}(Y_{A_1}) \cup \hdots ... \cup \operatorname{FV}(Y_{A_n})$.
	\begin{equation}
		\llbracket \Gamma. \langle Y_{A_1},...,Y_{A_n}\rangle\rrbracket:  \prod_{i = 1}^m \Omega^{\operatorname{cod}f_i} \lto \prod_{j = 1}^n \Omega^{A_j}
	\end{equation}
	We consider another term too; we let $Z_1:\call{P}A_1,...,Z_{m}:\call{P}A_n$ be variables and construct the following.
	\begin{equation}
		\langle f(Z_1),...,f(Z_m)\rangle
	\end{equation}
	This term induces a morphism, we let $\Theta$ denote the context $Z_1:\call{P}\operatorname{dom}f_1,...,Z_{m}:\call{P}\operatorname{dom}f_m$.
	\begin{equation}\label{eq:naive}
		\llbracket \Theta. \langle f_1(Z_1),...,f_m(Z_m)\rangle \rrbracket: \prod_{i = 1}^m \Omega^{\operatorname{dom}f_i} \lto \prod_{i = 1}^m \Omega^{\operatorname{cod}f_i}
	\end{equation}
	We then compose \eqref{eq:compression} and \eqref{eq:naive} to form the following morphism.
	\begin{equation}\label{eq:bottom_eq_map}
		\begin{tikzcd}[column sep = huge]
			\prod_{i = 1}^m \Omega^{\operatorname{dom}f_i}\arrow[rr,"{\llbracket \Theta. \langle f_1(Z_1),...,f_m(Z_m)\rangle\rrbracket}"] && \prod_{i = 1}^m \Omega^{\operatorname{cod}f_i}\arrow[d,"{\llbracket \Gamma. \langle Y_{A_1},...,Y_{A_n}\rangle\rrbracket}"]\\
			&& \prod_{j = 1}^n \Omega^{A_j}
		\end{tikzcd}
	\end{equation}
	We observe that the composite \eqref{eq:bottom_eq_map} can be described by a single term.
	\begin{align}
		\begin{split}
			&\llbracket \Theta. \langle f_1(Z_1),...,f_m(Z_m)\rangle \rrbracket \circ \llbracket \Gamma . \langle Y_{A_1},...,Y_{A_n}\rangle \rrbracket \\
		&= \llbracket \Theta. \langle f_{k_1^1}(Z_{k_1^1}) \cup \hdots \cup f_{k_{l_1}^1}(Z_{k_{l_1}^1}),...,f_{k_1^n}(Z_{k_1^n}) \cup \hdots \cup f_{k_{l_n}^n}(Z_{k_{l_n}^n})\rangle \rrbracket
		\end{split}
	\end{align}
	We give a name to this term.
	\begin{align*}
		t_1 := \langle f_{k_1^1}(Z_{k_1^1}) \cup \hdots \cup f_{k_{l_1}^1}(Z_{k_{l_1}^1}),...,f_{k_1^n}(Z_{k_1^n}) \cup \hdots \cup f_{k_{l_n}^n}(Z_{k_{l_n}^n})\rangle
	\end{align*}
	We now make the observation that the two contexts $\Delta$ and $\Theta$ are equal. In summary of what we have so far we write the following pair of morphisms.
	\begin{equation}
		\begin{tikzcd}
			\prod_{i = 1}^m \Omega^{\operatorname{dom}f_i}\arrow[rr,shift left, "{\llbracket\Theta.  t_0\rrbracket}"]\arrow[rr,swap,shift right,"{\llbracket \Delta. t_1\rrbracket}"] & & \prod_{j = 1}^n\Omega^{A_j} 
		\end{tikzcd}
	\end{equation}
	The next step is to take equivalence classes as per Section \ref{sec:coequaliser}. To do this, we first recall the term representing the relation $t_0(a) \sim t_1(a)$ of Definition \ref{def:relation}, and we consider the following term.
	\begin{equation}\label{eq:equiv_class}
		[b]_{t_0,t_1} = \Big\lbrace b':\prod_{i = 1}^n \call{P}A_i \mid \langle b,b'\rangle \in R_{t_0,t_1}\Big\rbrace
	\end{equation}
	We notice that the term \eqref{eq:equiv_class} is of type $\call{P}\big(\prod_{i = 1}^n \call{P}A_i\big)$ and has free variables $\operatorname{FV}([b]_{t_0,t_1}) = \lbrace b: \prod_{i = 1}^n \call{P}A_i\rbrace$. We now consider its interpretation.
	\begin{equation}
		\begin{tikzcd}[column sep = huge]
			\prod_{i =1 }^n \Omega^{A_i}\arrow[rr,"{\llbracket b:\prod_{i = 1}^n \call{P}A_i. [b]_{t_0,t_1} \rrbracket}"] & & \Omega^{\prod_{i = 1}^n \Omega^{A_i}}
		\end{tikzcd}
	\end{equation}
	To summarise, we now have the following Diagram.
	\begin{equation}\label{eq:colimit_parent}
		\begin{tikzcd}[column sep = huge]
			\prod_{i = 1}^m \Omega^{\operatorname{dom}f_i}\arrow[r,shift left, "{\llbracket\Theta.  t_0 \rrbracket}"]\arrow[r,swap,shift right,"{\llbracket \Delta. t_1\rrbracket}"]& \prod_{j = 1}^n\Omega^{A_j} \arrow[rr,"{\llbracket b:\prod_{i = 1}^n \call{P}A_i. [b]_{t_0,t_1} \rrbracket}"] & & \Omega^{\prod_{i = 1}^n \Omega^{A_i}}\\
			\coprod_{i = 1}^m \Omega^{\operatorname{dom}f_i}\arrow[u,rightarrowtail] & & & \operatorname{Coeq}(t_0,t_1)\arrow[u,rightarrowtail]
		\end{tikzcd}
	\end{equation}
	What we have constructed is the top row, completely described via terms in the Interal Logic of a topos, but we now want to show that this descends to a Diagram describing a colimit.
	
	We consider first the following fragment of Diagram \eqref{eq:colimit_parent}.
	\begin{equation}\label{eq:coproduct_parent}
		\begin{tikzcd}[column sep = huge]
			\prod_{i = 1}^m \Omega^{\operatorname{dom}f_i}\arrow[r,shift left, "{\llbracket\Theta.  t_0 \rrbracket}"]\arrow[r,swap,shift right,"{\llbracket \Delta. t_1\rrbracket}"] & \prod_{j = 1}^n\Omega^{A_j}\\
			\coprod_{i = 1}^m \Omega^{\operatorname{dom}f_i}\arrow[u,rightarrowtail] & \coprod_{j = 1}^n \Omega^{A_j}\arrow[u,rightarrowtail]
		\end{tikzcd}
	\end{equation}
	Sequent \eqref{seq:coproduct_decend} below is derivable in the type theory for $k = 0,1$. In \eqref{seq:coproduct_decend} the term $\langle \varnothing, \hdots, z_i, \hdots, \varnothing \rangle$ is the $m$-fold product with the $j^{\text{th}}$ term equal to $\varnothing_{\operatorname{dom}f_j}$ for $j \neq i$ and with $i^{\text{th}}$ term equal to $z_i$.
	\begin{equation}\label{seq:coproduct_decend}
		\vdash_{z_i : \Omega^{\operatorname{dom}f_i}} t_k\langle \varnothing, \hdots, z_i, \hdots, \varnothing \rangle \in \coprod_{j = 1}^n \Omega^{A_j}
	\end{equation}
	and so it follows from the fact that $\coprod_{i = 1}^m \Omega^{\operatorname{dom}f_i}$ is a coproduct (Proposition \ref{prop:coproduct}) that there exists the dashed morphisms in Diagram \eqref{eq:coproduct_decend}. The square created by deleting the bottom most morphism in both the top and bottom row of \eqref{eq:coproduct_decend} commutes, and similarly the square created by deleting the top most morphism in both the top and bottom row of \eqref{eq:coproduct_decend} commutes.
	\begin{equation}\label{eq:coproduct_decend}
		\begin{tikzcd}[column sep = huge]
			\prod_{i = 1}^m \Omega^{\operatorname{dom}f_i}\arrow[r,shift left, "{\llbracket\Theta.  t_0 \rrbracket}"]\arrow[r,swap,shift right,"{\llbracket \Delta. t_1\rrbracket}"] & \prod_{j = 1}^n\Omega^{A_j}\\
			\coprod_{i = 1}^m \Omega^{\operatorname{dom}f_i}\arrow[u,rightarrowtail]\arrow[r,dashed, shift left]\arrow[r,dashed, shift right] & \coprod_{j = 1}^n \Omega^{A_j}\arrow[u,rightarrowtail]
		\end{tikzcd}
	\end{equation}
	Turning now to the following Diagram of solid arrows.
	\begin{equation}\label{eq:coeq_parent}
		\begin{tikzcd}[column sep = huge]
			\prod_{j = 1}^n\Omega^{A_j} \arrow[rr,"{\llbracket b:\prod_{i = 1}^n \call{P}A_i. [b]_{t_0,t_1} \rrbracket}"] && \Omega^{\prod_{i = 1}^n \Omega^{A_i}}\\
			\coprod_{j = 1}^n \Omega^{A_j}\arrow[u,rightarrowtail]\arrow[rr,dashed] && \operatorname{Coeq}(t_0,t_1)\arrow[u,rightarrowtail]
		\end{tikzcd}
	\end{equation}
	The following Sequent is derivable in the type theory for $k=0,1$.
	\begin{equation}\label{seq:coeq_decend}
		\vdash_{w_i:\Omega^{\operatorname{dom}f_i}}[t_k\langle \varnothing, \hdots, w_i, \hdots, \varnothing \rangle] \in \operatorname{Coeq}(t_0,t_1)
	\end{equation}
	Note: In \eqref{eq:coproduct_decend} the term $\langle \varnothing, \hdots, z_i, \hdots, \varnothing \rangle$ is an $m$-fold product and in \eqref{eq:coeq_parent} the seemingly similar term $\langle \varnothing, \hdots, w_i, \hdots, \varnothing \rangle$ is an $n$-fold product.
	
	Again, using Proposition \ref{prop:coproduct} there exists the dashed morphism in Diagram \eqref{seq:coeq_decend} rendering the Diagram considered as a whole, commutative.
	
	To summarise, and this is indeed our main result, we now have the following commutative Diagram.
	\begin{equation}\label{eq:colimit_complete}
		\begin{tikzcd}[column sep = huge]
			\prod_{i = 1}^m \Omega^{\operatorname{dom}f_i}\arrow[r,shift left, "{\llbracket\Theta.  t_0 \rrbracket}"]\arrow[r,swap,shift right,"{\llbracket \Delta. t_1\rrbracket}"] & \prod_{j = 1}^n\Omega^{A_j} \arrow[rr,"{\llbracket b:\prod_{i = 1}^n \call{P}A_i. [b]_{t_0,t_1} \rrbracket}"] & & \Omega^{\prod_{i = 1}^n \Omega^{A_i}}\\
			\coprod_{i = 1}^m \Omega^{\operatorname{dom}f_i}\arrow[u,rightarrowtail]\arrow[r,shift left, dashed]\arrow[r,dashed, shift right] & \coprod_{j = 1}^n \Omega^{A_j}\arrow[u,rightarrowtail]\arrow[rr,dashed] & & \operatorname{Coeq}(t_0,t_1)\arrow[u,rightarrowtail]
		\end{tikzcd}
	\end{equation}

	\appendix
	\section{Topos theory and logic}\label{sec:topos_struc}
	What is a logical connective? One interpretation of connectives is offered by ZF set theory. The axiom of specification allows for the manipulation of sets via the manipulation of predicates. For instance, if $\varphi(a),\psi(a)$ are formulas with free variable $a$, then the set $\lbrace a\in A \mid \varphi(a) \wedge \psi(a)\rbrace$ is given by the intersection of the sets $\lbrace a \in A \mid \varphi(a)\rbrace \cap \lbrace a \in A \mid \psi(a)\rbrace$.
	
	In fact, this turns out to be a special case of a more general theory. The category $\underline{\operatorname{Sets}}$ is an elementary topos, and indeed any elementary topos $\call{E}$ offers a perspective on connectives. We show in this Section how the logical structure of the connectives $\forall, \exists, \wedge, \vee, \Rightarrow$ is captured by the categorical structure of $\call{E}$.
	\subsection{Quantifiers}
	\begin{definition}\label{def:image}
		Let $f: A \lto B$ be a morphism in a topos $\call{E}$. We consider two copies of $f$ and construct the pushout.
		\begin{equation}
			\begin{tikzcd}
				A\arrow[r,"{f}"]\arrow[d,swap,"{f}"] & B\arrow[d,"{\iota_1}"]\\
				B\arrow[r,swap,"{\iota_2}"] & C
			\end{tikzcd}
		\end{equation}
		We then define the \textbf{image} of $f$ to be the Equaliser of $\iota_1,\iota_2$.
		\begin{equation}
			\operatorname{im}f := \operatorname{Equal}(\iota_1,\iota_2)
		\end{equation}
	\end{definition}
	\begin{theorem}
		\label{adj}
		Let $\call{E}$ be a topos and $f:A \to B$ a morphism. Then the functor $f^{-1}: \text{Sub}(B) \to \text{Sub}(A)$ admits both a left and a right adjoint.
	\end{theorem}
	\begin{proof}
		See Johnstone \cite[\S A 1.4.10]{Johnstone}.
	\end{proof}
	\begin{definition}
		\label{qunatifiers}
		Let $f:A \to B$ be a morphism in a topos. Then the left adjoint to $f^{-1}$ will be denoted $\exists_f$, and the right adjoint by $\forall_f$.
	\end{definition}
	The reason why this notation is used, is because in the topos \underline{Set}, these adjoints are given by the following explicit maps
	\begin{gather*}
		\exists_f: \text{Sub}(A) \to \text{Sub}(B)\\
		A' \mapsto \lbrace b \in B \mid \text{ there exists }a \in A'\text{ such that }f(a) = b\rbrace
	\end{gather*}
	and
	\begin{gather*}
		\forall_f: \text{Sub}(A) \to \text{Sub}(B)\\
		A' \mapsto \lbrace b \in B \mid \text{ for all }a \in A\text{ if }f(a) = b\text{ then }a \in A'\rbrace
	\end{gather*}
	\begin{proof}
		See \cite[\S1 9.2]{MM}.
	\end{proof}
	In the general context of an arbitrary topos, the left adjoint is particularly easy to describe.
	\begin{lemma}
		Let $f:A \lto B$ be a morphism and let $m: M \rightarrowtail A \in \operatorname{Sub}(A)$ be a subobject of $A$.
		\begin{equation}
			\exists_f(m) \cong \operatorname{im}(fm)
		\end{equation}
	\end{lemma}
	\begin{proof}
		We will define a functor $\operatorname{im}(f\und{0.2}): \operatorname{Sub}A \lto \operatorname{Sub}B$ and prove that this functor is left adjoint to the functor $f^{-1}: \operatorname{Sub}B \lto \operatorname{Sub}A$, the result will then follow from essential uniqueness of adjoints.
		
		We define the image of a subobject $m: M \rightarrowtail A$ under $\operatorname{im}(f\und{0.2})$ to be the subobject $\operatorname{im}(fm)$ of $B$. Now we define how this functor behaves on morphims.
		
		Let $m: M \rightarrowtail A, n:N \rightarrowtail A$ be subobjects of $A$ and let $\gamma: M \lto N$ be a morphism such that the $m = n\gamma$. Let $i_1,i_2$ denote the morphisms such that $\operatorname{im}(fm) = \operatorname{Equal}(i_1,i_2)$ and let $i_1',i_2'$ denote the morphisms such that $\operatorname{im}(fn) = \operatorname{Equal}(i_1',i_2')$. To define a morphism $\operatorname{im}(fm) \lto \operatorname{im}(fn)$ it suffices to define a morphism $\gamma: C \lto C'$ such that $\gamma i_1 = i_1'$ and $\gamma i_2 = i_2'$. For convenience, we draw the following Diagram, we have denoted by $C$ the object in the pushout $\operatorname{PushOut}(fm,fm)$ and $C'$ for that of $\operatorname{PushOut}(fn,fn)$.
		\begin{equation}
			\begin{tikzcd}[column sep = huge, row sep = huge]
				\operatorname{im}(fm)\arrow[dr] & & C\arrow[d,dashed,"{\gamma'}"]\\
				\operatorname{im}(fn)\arrow[r] & B\arrow[ur,shift left,"{i_1}"]\arrow[ur,shift right,swap,"{i_2}"]\arrow[r,shift left,"{i_1'}"]\arrow[r,shift right,swap,"{i_2'}"] & C'
			\end{tikzcd}
		\end{equation}

		We consider the following Diagram of solid arrows.
		\begin{equation}
			\begin{tikzcd}[column sep = huge]
				N\arrow[rrd, bend left,"{fn}"]\arrow[ddr,swap,bend right,"{fm}"]\\
				& M\arrow[r,"{fm}"]\arrow[d,swap,"{fm}"]\arrow[ul,"{\gamma}"] & B\arrow[d,"{i_1}"]\arrow[ddr,bend left,"{i_1'}"]\\
				& B\arrow[r,swap,"{i_2}"]\arrow[drr,swap,bend right,"{i_2'}"] & C\arrow[dr,dashed,"{\gamma'}"]\\
				& & & C'
			\end{tikzcd}
		\end{equation}
		Since $\operatorname{im}(fm)$ is a pushout, the existence of $\operatorname{im}(f\gamma)$ which renders the Diagram commutative follows from the fact that $fm = fn \gamma$, which follows from the assumption that $m = n \gamma$.
		
		Now we need to show that $\operatorname{im}(f\und{0.2})$ is left adjoint to $f^{-1}$. That is, given $m: M \rightarrowtail A \in \operatorname{Sub}A, r: R \rightarrowtail B \in \operatorname{Sub}B$ along with a morphism $\delta:m \lto f^{-1}(r)$, we need to show that there exists a morphism $\eta_m: m \lto f^{-1}\operatorname{im}(fm)$, natural in $m$, and a unique morphism $\beta:\operatorname{im}(fm) \lto r$ such that the following Diagram commutes.
		\begin{equation}
			\begin{tikzcd}
				m\arrow[r,"{
					\eta_{m}}"]\arrow[dr,swap,"{\delta}"] & f^{-1}\operatorname{im}(fm)\arrow[d,"{f^{-1}\beta}"]\\
				& f^{-1}(r)
			\end{tikzcd}
		\end{equation}
		We make a note that $R \cong \operatorname{im}r$, as $r$ is monic, so it suffices to show the above statement with every instance of $r$ replaced by $\operatorname{im}r$. We overload the notation $r$ and use this to denote the morphism $\operatorname{im}r \rightarrowtail B$ as well as the morphism $R \rightarrowtail B$.
		
		We let $i_1,i_2: B \lto C$ denote the morphisms such that $\operatorname{im}(fm) = \operatorname{Equal}(i_1,i_2)$, let $i_1',i_2': B \lto C'$ denote the morphisms such that $\operatorname{im}(r) = \operatorname{Equal}(i_1',i_2')$. First we define $\eta_m: m \lto f^{-1}\operatorname{im}(fm)$. The natural morphism $\eta_m: M \lto f^{-1}\operatorname{im}(fm)$ is given by pulling back $\operatorname{im}(fm)$ along $f$ and using the fact that $fm$ factors through $\operatorname{im}(fm)$. Naturally follows easily from universality of the objects involved. Now we show the eixstence of the morphism $\beta: \operatorname{im}(fm) \lto r$. Consider the following Diagram of solid arrows.
		\begin{equation}\label{eq:adj_proof}
			\begin{tikzcd}[column sep = huge, row sep = huge]
				&&& C'\\
				&&& C\arrow[u,dashed,"{\zeta}"] \\
				&A\arrow[rr,"{f}"] && B\arrow[u,shift left,"{i_1}"]\arrow[u,shift right, swap,"{i_2}"]\arrow[uu,bend left, "{i_1'}"]\arrow[uu,bend right,swap,"{i_2'}"]\\
				M\arrow[ur,rightarrowtail,"{m}"]\arrow[r,swap,"{\delta}"] & f^{-1}(\operatorname{im}r)\arrow[u,rightarrowtail]\arrow[rr,bend right,swap,"{\alpha}"] & \operatorname{im}(fm)\arrow[ur,rightarrowtail]\arrow[r,dotted,swap,"{\beta}"] & \operatorname{im}r\arrow[u,rightarrowtail,swap,"{r}"]
			\end{tikzcd}
		\end{equation}
		Now, by the universal property of the equaliser $\operatorname{im}r$, to show the existence of an appropriate morphism $\operatorname{im}(fm) \lto \operatorname{im}r$ it suffices to define a morphism $\zeta: C \lto C'$, denoted by a dashed arrow in \eqref{eq:adj_proof}, rendering \eqref{eq:adj_proof} commutative. The objects $C,C'$ are respectively objects of pushouts, and so in turn it suffices to prove that $i_1' f m = i_2' f m$. This follows immediately from the fact that $i_1' r \alpha \delta = i_2' r \alpha \delta$.

	\end{proof}
	
	\subsection{Heyting algebras}
	\begin{definition}
		A \textbf{Heyting algebra} is a set with a partial order, which as a category admits the following.
		\begin{itemize}
			\item Initial and terminal objects.
			\item Binary products, binary coproducts.
			\item All exponentials.
		\end{itemize}
	\end{definition}
	Many examples are given by the following Theorem.
	\begin{theorem}
		\label{thm:heyting_sub}
		Let $\call{E}$ be a topos, and let $E$ be any object of $\call{E}$. Then the set $\text{Sub}(E)$, with partial order given by inclusion, is a Heyting algebra with the relevant structure given as follows.
		\begin{itemize}
			\item Terminal object given by $\text{id}_E: E \to E$.
			\item Initial object given by the monic $0 \rightarrowtail E$, where $0$ is the initial object of the topos $\call{E}$. Note; this morphism $0 \rightarrowtail E$ is necessarily monic as $\call{E}$ is a topos (see \cite[\S A1.4.1]{Johnstone}).
			\item The binary product of subobjects $m: A \rightarrowtail E$ and $n: B \rightarrowtail E$ is given by the triple $(A \times_{E} B,\pi_A,\pi_B)$, which is such that the following is a pullback diagram in $\call{E}$.
			\[
			\begin{tikzcd}
				A \times_E B\arrow[r,"{\pi_A}"]\arrow[d,swap,rightarrowtail,"{\pi_B}"] & A\arrow[d,rightarrowtail,"m"]\\
				B\arrow[r,swap,rightarrowtail,"{n}"] & E
			\end{tikzcd}
			\]
			\item The binary coproduct of subobjects $m: A \rightarrow E$ and $n: B \rightarrow E$ given by the triple $(\operatorname{Im}(m\coprod n), k\iota_A,k\iota_B')$ where $(A \coprod B, \iota_A,\iota_B)$ is a coproduct in $\call{E}$, $k$ is the induced morphism $M \coprod N \lto \operatorname{im}(m \coprod n)$, and $m \coprod n$ is the unique morphism $A \coprod B \to E$ such that the following diagram commutes in $\call{E}$.
			\[
			\begin{tikzcd}[row sep = huge]
				& E\\
				A\arrow[ur,rightarrowtail,"{m}"]\arrow[r,swap,rightarrowtail,"{\iota_A}"] & A \coprod B\arrow[u,"{m \coprod n}"] & B\arrow[ul,swap,rightarrowtail,"{n}"]\arrow[l,rightarrowtail,"{\iota_B}"]
			\end{tikzcd}
			\]
			\item The exponentials are more complicated, see \cite[\S A1.4.13]{Johnstone}.
		\end{itemize}
	\end{theorem}
	\begin{proof}
		See \cite[\S A1.4]{Johnstone}.
	\end{proof}
	\begin{definition}
		\label{heytingnotation}
		We denote the initial and terminal objects of a Heyting algebra by $0$ and $1$ respectively. The object corresponding to the binary product of objects $A$ and $B$ will be denoted $A \wedge B$, and the object corresponding the coproduct, $A \vee B$. Lastly, the objects $B^A$ will be denoted $A \Rightarrow B$.
	\end{definition}
	\begin{remark}
		\label{rmk:countable_colimits}
		If $\call{E}$ admits countably infinite coproducts, then for any object $E \in \call{E}$, the category $\operatorname{Sub}(E)$ also admits arbitrary coproducts. See \cite[\S A1.4]{Johnstone}. If $\lbrace p_i \rbrace_{i \geq 0}$ is a countably infinite set of elements of $\operatorname{Sub}(E)$ then the coproduct will be denoted by $\bigvee_{i\geq 0}p_i$.
	\end{remark}
	
	\refs
		\bibitem [May, 1999]{May} J. May, \emph{A Crash Course in Algebraic Topology}, University of Chicago Press, Chicago, 1999.
		\bibitem [Borceux, 1994]{borceux} F. Borceux, \emph{Handbook of Categorical Algebra 1, Basic Category Theory}, University Press, Cambridge, 1994.
		\bibitem [Godel, 1931]{Godel} K. Godel, \emph{Uber formal unentscheidbare Sätze der Principia Mathematica und verwandter Systeme I}, Monatshefte für Mathematik und Physik, 38: 173–198.
		\bibitem [Turing, 1936]{Turing} A. Turing, \emph{On Computable Numbers, with an Application to the Entscheidungsproblem}, Proceedings of the London Mathematical Society, 1936.
		\bibitem [Church, 1936]{Church} A. Church, \emph{A Note on the Entscheidungsproblem}, New York: The Association for Symbolic Logic, Inc., 1936
		\bibitem [Lawvere, 1965]{Lawvere} W. Lawvere, \emph{An Elementary Theory of the Category of Sets}, Proceedings of
		the National Academy of Science of the U.S.A 52, 1506–1511
		\bibitem [Russell, 1903]{russell} B. Russell, A. Whitehead, \emph{The Principles of Mathematics}, Cambridge, Cambridge University Press, 1903,
		\bibitem [Sorensen, Urzyczyn, 2006]{su} M Sorensen, P. Urzyczyn, \emph{Lectures on the Curry-Howard Isomorphism}
		\bibitem [Meordijk, Mac Lane, 1994]{MM} I. Meordijk, S. Mac Lane, \emph{Sheaves in Geometry and Logic}, Springer-Verlag, New York, 1992.
		\bibitem [Lambek, Scott, 1986]{lambekscott} J. Lambek, P.J. Scott, \emph{Introduction to Higher Order Categorical Logic}, Cambridge University Press, New York, 1986.
		\bibitem [Johnstone, 2002]{Johnstone} P. Johnstone, \emph{Sketches of an Elephant; A Topos Theory Compendium}, Clarendon Press, Oxford, 2002
		\bibitem [Mac Lane, 1969]{MacLane} S. MacLane, \emph{Homology}, Springer-Verlag, 1969
		\bibitem [Troiani, 2019]{troiani} W. Troiani, \emph{Simplicial sets are algorithms}, (Masters Thesis) \url{http://therisingsea.org/notes/MScThesisWillTroiani.pdf}
		\bibitem [Mikkelson, 1976]{Mikkelson} C. J. Mikkelson \emph{Lattice Theoretic and Logical Aspects of Elementary Topoi}, Aarhus Universitet, Matematisk Institut (January 1, 1976).
		\bibitem [Pare, 1974]{Pare} R. Paré, \emph{Colimits in Topoi}, Bulletin of the American Mathematical Society, 1974.
		
	\endrefs
\end{document}